\documentclass[letterpaper,reqno,12pt]{amsart}
\usepackage[margin=1.2in]{geometry}

\usepackage{amssymb, enumerate, color}
\usepackage{mathrsfs}
\usepackage[all]{xy}
\usepackage{hyperref}
\usepackage{enumitem}
\usepackage{graphicx}  

\usepackage{amsmath}

\numberwithin{equation}{section} 

\theoremstyle{plain}
\newtheorem{thm}{Theorem}[section] 
\newtheorem{cor}[thm]{Corollary}
\newtheorem{prop}[thm]{Proposition}

\newtheorem{lem}[thm]{Lemma}

\theoremstyle{definition} 
\newtheorem{defn}[thm]{Definition}
\newtheorem{lem-defn}[thm]{Lemma-Definition}
\newtheorem{setting}[thm]{Setting}
\newtheorem{eg}[thm]{Example} 

\newtheorem*{notation}{Notation}

\theoremstyle{remark}
\newtheorem{rem}[thm]{Remark}

\newtheorem*{acknowledgement}{Acknowledgments}

\def\ge{\geqslant}
\def\le{\leqslant}
\def\phi{\varphi}
\def\epsilon{\varepsilon}

\def\bar{\overline}
\def\mapsto{\longmapsto}

\newcommand{\sO}{\mathcal{O}}

\newcommand{\F}{\mathbb{F}}
\newcommand{\N}{\mathbb{N}}
\newcommand{\Q}{\mathbb{Q}} 
\newcommand{\C}{\mathbb{C}} 
\newcommand{\R}{\mathbb{R}} 
\newcommand{\Z}{\mathbb{Z}}

\newcommand{\ba}{\mathfrak{a}}

\newcommand{\m}{\mathfrak{m}}
\newcommand{\n}{\mathfrak{n}}
\newcommand{\p}{\mathfrak{p}}
\newcommand{\q}{\mathfrak{q}}

\newcommand{\adj}{\mathrm{adj}}

\newsavebox{\circlebox}
\savebox{\circlebox}{\fontencoding{OMS}\selectfont\Large\char13}
\newlength{\circleboxwdht}

\def\Hom{\operatorname{Hom}}
\def\Spec{\operatorname{Spec}}
\def\Proj{\operatorname{Proj}}
\def\Supp{\operatorname{Supp}}

\def\Div{\operatorname{div}}

\def\ord{\operatorname{ord}}

\def\Tr{\operatorname{Tr}}

\def\B{\mathcal{B}}
\DeclareMathOperator*{\ulim}{ulim}
\def\Ann{\operatorname{Ann}}
\def\sL{\mathcal{L}}
\def\Ht{\operatorname{ht}}
\def\J{\mathcal{J}}

\title{On the behavior of adjoint ideals under pure morphisms}

\author{Shunsuke Takagi}
\address{Graduate School of Mathematical Sciences, University of Tokyo, 3-8-1 Komaba, Meguro-ku, Tokyo 153-8914, Japan}
\email{stakagi@ms.u-tokyo.ac.jp}

\author{Tatsuki Yamaguchi}
\address{Department of Mathematics, Institute of Science Tokyo, 2-12-1 Ookayama, Meguro-ku, Tokyo 152-8551, Japan}
\email{yamaguchi.t.bp@m.titech.ac.jp}

\subjclass[2020]{Primary 14B05; Secondary 03C20, 13A35, 14F18}

\dedicatory{Dedicated to Professor Keiji Oguiso on the occasion of his sixtieth birthday.}

\begin{document}

\begin{abstract}
We characterize adjoint ideal sheaves via ultraproducts and study their behavior under pure morphisms by utilizing this characterization. 
In particular, given a pure morphism $f:Y \to X$ between normal quasi-projective complex varieties, a reduced divisor $D$ and an effective $\Q$-Weil divisor $\Gamma$ on $X$ with no common components, we have the following result: 
if the cycle-theoretic pullback $E:=f^{\natural}D$ is reduced and $(Y, E+f^*\Gamma)$ is of plt type along $E$, then $(X, D+\Gamma)$ is of plt type along $D$. 
This provides an affirmative answer to a question posed by Z.~Zhuang \cite{Zhu}. 
\end{abstract}

\maketitle
\markboth{S.~TAKAGI and T.~YAMAGUCHI}{ADJOINT IDEALS AND PURE MORPHISMS}

\tableofcontents

\section{Introduction}
A ring extension $R \hookrightarrow S$ is said to be \textit{pure} (also known as universally injective) if the natural map $M=M \otimes_R R \to M \otimes_R S$ is injective for every $R$-module $M$. 
For example, faithfully flat ring extensions are pure. 
Another important example of pure ring extensions occurs when a linearly reductive group $G$ acts on an affine variety $Y$ and $\pi:Y \to X:=Y//G$ is the quotient morphism; in this case, the associated ring homomorphism $\sO_X \to \pi_*\sO_Y$ is pure.  
Given a pure ring extension $R \hookrightarrow S$, a natural and classical question to ask is which  geometric properties descend from $S$ to $R$. 
In this paper, we explore this question for $R$ and $S$ that are essentially of finite type over the field $\C$ of complex numbers. 

A first result in this direction is due to Boutot \cite{Bo}. 
He proved that if $\Spec S$ has rational singularities, so does $\Spec R$. 
When $R \hookrightarrow S$ splits, Kov\'acs \cite{Kov} proved that if $\Spec S$ has Du Bois singularities, so does $\Spec R$. 
More recently, Godfrey and Murayama \cite{GM} extended this result without assuming the splitting.  
In this paper, we devote our attention to singularities appearing in the minimal model program, namely klt, plt, and lc singularities. 
Classical singularities in the minimal model program are $\Q$-Gorenstein by definition, and Schoutens \cite{Sch05} showed, under the assumption that $R$ and $S$ are both $\Q$-Gorenstein, that if $\Spec S$ has klt singularities, then $\Spec R$ does as well. 
However, $R$ is not necessarily $\Q$-Gorenstein even if $S$ is regular, and consequently, the property of being (classical) klt or lc does not descend under pure ring extensions in general.
On the other hand, de Fernex-Hacon \cite{dFH} generalized the definitions of klt and lc singularities to the non-$\mathbb{Q}$-Gorenstein setting, which are called singularities of klt and lc types, respectively (see Definition \ref{adjoint def} for their definitions). 
 Our study focuses on these generalized singularities.
Braun et al.~\cite{BGLM} showed that if  $\Spec S$ is of klt type and a linearly reductive group $G$ acts on $S$, then $\Spec S^G$ is also of klt type, where $S^G$ is the subring of invariants under the action of $G$. 
As mentioned above, $S^G$ is a pure subring of $S$.  
Very recently, Zhuang generalized their result as follows: 

\begin{thm}[\cite{Zhu}]\label{Zhuang thm}
If $\Spec S$ is of klt type, then so is $\Spec R$. 
\end{thm}

The main result of this paper is a generalization of Theorem \ref{Zhuang thm} formulated in terms of adjoint ideals. 
To state this, we first explain plt type singularities and adjoint ideals. 
Let $D$ be a reduced divisor and $\Gamma$ be an effective $\R$-Weil divisor on a normal quasi-projective variety $X$ such that $D$ and $\Gamma$ have no common components.  
When $K_X+D+\Gamma$ is $\R$-Cartier, the pair $(X, D+\Gamma)$ is said to be \textit{plt} along $D$ if its discrepancy at $E$ is greater than $-1$ for every prime divisor $E$ over $X$ that is not an irreducible component of the strict transform of $D$. 
When $K_X+D+\Gamma$ is not necessarily $\R$-Cartier, we say that $(X, D+\Gamma)$ is of \textit{plt type} along $D$ if there exists an effective $\R$-Weil divisor $\Delta$ on $X$ with no common components with $D$ such that $K_X+D+\Gamma+\Delta$ is $\R$-Cartier and $(X, D+\Gamma+\Delta)$ is plt along $D$. 
Note that when $D=0$, being of plt type along $D$ is nothing but being of klt type. 
Our adjoint ideal $\adj_D(X, D+\Gamma)$ is a variant of multiplier ideals and defines the non-plt-type locus of the pair $(X, D+\Gamma)$ along $D$ (see Definition \ref{adjoint def} for the precise definition).  
We obtain the following results on the behavior of adjoint ideals under pure morphisms. 
Since $\adj_D(X, D+\Gamma)$ is nothing but the multiplier ideal $\J(X, \Gamma)$ when $D=0$, this gives a generalization of \cite[Theorem 1.2]{Yam}. 

\begin{thm}[Theorem \ref{faithfully flat thm}, Corollary \ref{adjoint ideal under pure morphism in Q-Cartier case}, Theorem \ref{adjoint ideals under pure ring extensions in rings with finitely generated anti-canonical algebra}]\label{main thm1}
Let $R \hookrightarrow S$ be a pure ring extension of normal domains of finite type over $\C$ and $Y:=\Spec S \xrightarrow{f} X:=\Spec R$ denote the corresponding morphism of affine varieties. 
Let $D$ be a reduced divisor and $\Gamma$ be an effective $\Q$-Weil divisor on $X$ such that $D$ and $\Gamma$ have no common components. 
Suppose that the $\sO_X$-algebra $\bigoplus_{i \ge 0}\sO_X(\lfloor -i(K_X+D+\Gamma) \rfloor)$ is finitely generated and the cycle-theoretic pullback $E:=f^{\natural}D$ of $D$ is a reduced divisor on $Y$ (see Definition \ref{pullback} for the definitions of cycle-theoretic pullback). 

\begin{enumerate}
\item If $f$ is faithfully flat, then 
\[
\adj_E(Y, E+f^*\Gamma) \subseteq \adj_D(X, D+\Gamma) \sO_Y.
\]
\item Assume that $E$ is a disjoint union of prime divisors and that one of the following conditions holds. 
\begin{enumerate}
\item $K_X+D+\Gamma$ is $\Q$-Cartier. 
\item The $\sO_Y$-algebra $\bigoplus_{i \ge 0}\sO_Y(iB)$ is finitely generated for every reduced divisor  $B$ on $Y$.\footnote{This condition is satisfied, for example, if $Y$ is of klt type or $\Q$-factorial.}
\end{enumerate} 
Then 
\[
\adj_E(Y, E+f^*\Gamma) \cap \sO_X \subseteq \adj_D(X, D+\Gamma).
\]
\end{enumerate}
\end{thm}

Noting that the assumption of Theorem \ref{main thm1} (2) is satisfied when $(Y, E+f^*\Gamma)$ is of plt type along $E$, we have the following corollary, which generalizes Zhuang's result \cite[Theorem 2.10]{Zhu} and provides an affirmative answer to his question \cite[Question 2.13]{Zhu} in the klt case. 

\begin{cor}[Corollary \ref{cor plt case}]
Let $f:Y\to X$ be a pure morphism between normal quasi-projective complex varieties, 
$D$ be a reduced divisor and $\Gamma$ be an effective $\Q$-Weil divisor on $X$ such that $D$ and $\Gamma$ have no common components.  
Suppose that the cycle-theoretic pullback $E=f^\natural D$ of $D$ under $f$ is a reduced divisor on $Y$. 
If $(Y,E+f^* \Gamma )$ is of plt type along $E$, then $(X,D+\Gamma)$ is of plt type along $D$. 
In particular, if $(Y, f^* \Gamma )$ is of klt type, then $(X,\Gamma)$ is of klt type as well.
\end{cor}

We give a sketch of the proof of Theorem \ref{main thm1} below.
The proof relies primarily on the theory of test ideals as the main tool.
Originally introduced by Hochster-Huneke \cite{HH90} in the context of tight closure theory, test ideals have been shown to have a strong connection with multiplier ideals. This connection has inspired their generalization to the setting of log pairs.
When $K_X+D+\Gamma$ is $\Q$-Cartier, Takagi \cite{Tak08} proved that the adjoint ideal $\adj_D(X, D+\Gamma)$ of $(X, D+\Gamma)$ coincides, after reduction to characteristic $p \gg 0$,  with a variant of the generalized test ideal associated to $(X, D+\Gamma)$. 
We generalize this coincidence to the case where $K_X+D+\Gamma$ is not necessarily $\Q$-Cartier but the $\sO_X$-algebra $\bigoplus_{i \ge 0}\sO_X(\lfloor -i(K_X+D+\Gamma) \rfloor)$ is finitely generated. 
Then, since flatness is preserved under reduction modulo $p$, 
we can reduce Theorem \ref{main thm1} (1) to a problem on test ideals. 

The proof of Theorem \ref{main thm1} (2) is more involved because purity is not preserved under reduction modulo $p$ (see \cite{HJPS}).  
Therefore, we utilize ultraproducts rather than reduction modulo $p$. 
Using ultraproducts, Schoutens \cite{Sch04} constructed a ``canonical" big Cohen-Macaulay algebra $\mathcal{B}(R)$ over a normal local domain $R$ essentially of finite type over $\C$. 
As a generalization of the classical test ideal, 
P\'erez and R.~G. \cite{PR} introduced the notion of the BCM test ideal $\tau_B(R)$ associated to a big Cohen-Macaulay algebra $B$ (see also \cite{MS}), and Yamaguchi \cite{Yam} showed that the BCM test ideal $\tau_{\mathcal{B}(R)}(R)$ associated to $\mathcal{B}(R)$ is equal to the multiplier ideal $\J(\Spec R)$ if $R$ is $\Q$-Gorenstein. 
We introduce a new generalization $\tau_{\B, D}(R, D+\Gamma)$ of the ideal $\tau_{\mathcal{B}(R)}(R)$ for a prime divisor $D$ and an effective $\Q$-Weil divisor $\Gamma$ on $X:=\Spec R$ with no components supported on $D$. 
We then prove that this ideal is equal to the adjoint ideal $\adj_D(X, D+\Gamma)$ if $K_X+D+\Gamma$ is $\Q$-Cartier. 
Theorem \ref{main thm1} (2) (a) follows from this characterization of adjoint ideals. 
Specifically, by this characterization, the problem is reduced to showing the inclusion $\tau_{\B, E}(S, E+f^*\Gamma) \cap R \subseteq \tau_{\B, D}(R, D+\Gamma)$ of BCM test ideals, which is essentially true by definition.  
To illustrate the proof, we consider the case where  $D=E=\Gamma=0$ as a test case. 
In this situation, $\tau_{\B, D}(R, D+\Gamma)=\tau_{\B(R)}(R)=\bigcap_{M} \Ann_R 0^{\B(R)}_M$, where $M$ runs through all $R$-modules and $0^{\B(R)}_M$ is the kernel of the natural map $M \to M \otimes_R \B(R)$. 
Similarly, $\tau_{\B, E}(S, E+f^*\Gamma)=\tau_{\B(S)}(S)=\bigcap_{N} \Ann_S 0^{\B(S)}_N$, 
where $N$ runs through all $S$-modules.
For every $R$-module $M$, we have the following commutative diagram:  
\[
\xymatrix{
M \ar[r] \ar[d] & M \otimes_R S \ar[d] & \\
M \otimes_R \B(R) \ar[r] &  M \otimes_R \B(S) \ar@{=}[r] & (M \otimes_R S) \otimes_S \B(S),
}
\]
where the top horizontal map is injective by the purity of $R \hookrightarrow S$. 
The commutativity of this diagram implies that $0^{\B(R)}_M \subseteq 0^{\B(S)}_{M \otimes_R S}$, and consequently, 
\[
\tau_{\B(R)}(R)  \supseteq \bigcap_{M} \Ann_R 0^{\B(S)}_{M \otimes_R S}
=\bigcap_{M} \left(\Ann_S 0^{\B(S)}_{M \otimes_R S} \right) \cap R
\supseteq \tau_{\B(S)}(S) \cap R.
\]
This completes the proof of the inclusion of BCM test ideals. 

Finally, we briefly discuss the proof of Theorem \ref{main thm1} (2) (b). 
Following an idea of Zhuang \cite{Zhu},  we take some affine open subsets of $\Q$-Cartierizations to reduce the problem to the $\Q$-Cartier case.  
The $\Q$-Cartierization of a Weil divisor $B$ exists if the Rees algebra $\bigoplus_{i \ge 0} \sO(iB)$ is finitely generated. 
This is the reason why we assume the finite generation of certain Rees algebras in Theorem \ref{main thm1} (2) (b). 

Zhuang asked in \cite[Question 2.11]{Zhu} whether an analog of  Theorem \ref{Zhuang thm} holds for singularities of lc type. 
As another application of Theorem \ref{main thm1} (2), 
we give a partial affirmative answer to this question.  
\begin{thm}[Theorem \ref{lc case}]
Let $\phi:Y\to X$ be a pure morphism between normal complex affine varieties and let $\Gamma$ be an effective $\Q$-Weil divisor on $X$ such that $K_X+\Gamma$ is $\Q$-Cartier. 
Assume that one of the following conditions holds. 
\begin{enumerate}[label=(\roman*)]
\item There exists an effective $\R$-Weil divisor $\Delta$ on $Y$ such that $K_Y+\phi^*\Gamma+\Delta$ is $\Q$-Cartier and no non-klt center of $(Y, \phi^*\Gamma+\Delta)$ dominates $X$. 
\item The non-klt-type locus of $(Y,\phi^*\Gamma)$ has dimension at most one. 
\end{enumerate}
If $(Y,\phi^*\Gamma)$ is of lc type, then $(X,\Gamma)$ is lc.  
\end{thm}
Here, the pair $(Y,\phi^*\Gamma)$ is said to be of \textit{lc type} if there exists an effective $\R$-Weil divisor $\Delta$ on $Y$ such that $K_Y+\phi^*\Gamma+\Delta$ is $\R$-Cartier and $(Y,\phi^*\Gamma+\Delta)$ is lc (see Definition \ref{adjoint def} (3)). 


\begin{small}
\begin{acknowledgement}
The authors would like to thank Osamu Fujino, Kenta Sato and Chenyang Xu for valuable discussions. 
They are also grateful to the anonymous referee for their careful reading and numerous helpful suggestions.
The first author was partially supported by JSPS KAKENHI Grant Numbers 16H02141, 17H02831 and 20H00111 and the second author was partially supported by JSPS KAKENHI
Grant Number JP22J13150. 
\end{acknowledgement}
\end{small}

\begin{notation}
Throughout this paper, all rings are assumed to be commutative and with unit element and all schemes are assumed to be Noetherian and separated. 
All varieties are assumed to be quasi-projective. 
\end{notation}

\section{Preliminaries}

This section provides preliminary results needed for the rest of the paper. 

\subsection{Test ideals along divisors}
In this subsection, we recall the definition of test ideals along divisors.\footnote{Test ideals along divisors are referred to as divisorial test ideals in \cite{Tak08}.} 
The reader is referred to \cite{Tak08} and \cite{Tak13} for details.  
We will freely use the notation in \cite{TW}.

Suppose that $R$ is a normal domain of characteristic $p>0$ and $D$ is a reduced divisor on $X:=\Spec R$.  
Let $R^{\circ, D}$ denote the set of elements of $R$ not in any minimal prime of $I_D:=R(-D)$. 
Let $\Delta$ be an effective $\Q$-Weil divisor on $X$ that has no common components with $D$, let $\ba$ be an ideal of $R$ such that $\ba \cap R^{\circ, D} \ne \emptyset$ and $t \ge 0$ be a real number. 
We assume that $R$ is \textit{$F$-finite}, that is, the Frobenius map $F: R \to R$ is finite. 
This is equivalent to saying that the Frobenius pushforward $F_*R$ is a finitely generated $R$-module. 

\begin{defn}[cf.~\textup{\cite[Proposition 1.1]{Tak13}}]
The test ideal $\tau_D(R, D+\Delta, \ba^t)$ of the triple $(R, D+\Delta,\ba^t)$ along $D$ is defined as the unique smallest ideal $J$ of $R$ satisfying the following conditions:
\begin{enumerate}[label=(\alph*)]
\item $J \cap R^{\circ, D} \ne \emptyset$. 
\item For every integer $e \ge 0$ and every $\varphi \in \Hom_R(F^e_*R(\lceil (p^e-1)(D+\Delta) \rceil), R) \subseteq \Hom_R(F^e_*R, R)$, one has $\varphi(F^e_*(\ba^{\lceil t(p^e-1) \rceil}J)) \subseteq J$. 
\end{enumerate}
It is simply denoted by $\tau_D(R, D+\Delta)$ when $\ba=R$. 
\end{defn}

\begin{defn} \label{tight closure along D}
Suppose that $(R, \m)$ is local. 
For an $R$-module $M$, $0^{*_D(D+\Delta, \ba^t)}_{M}$ is the submodule of $M$ consisting of all elements $z \in M$ for which there exists an element $c \in R^{\circ, D}$ such that 
\[F^e_*(c \ba^{\lceil tp^e \rceil}) \otimes z=0 \in F^e_*R((p^e-1)D+\lceil p^e \Delta \rceil) \otimes_R M\]
for all large $e$. It is simply denoted as $0^{*_D(D+\Delta)}_{M}$ when $\ba=R$.
\end{defn}

\begin{prop}[cf.~\textup{\cite[Proposition 8.23]{HH90}, \cite[Theorem 6.3]{Schw10}}]\label{test ideal equiv def}
When $(R,\mathfrak{m})$ is local, the following ideals are equal to each other. 
\begin{enumerate}[label=$(\arabic*)$]
\item $\tau_{D}(R,D+\Delta,\ba^t)$. 
\item $\bigcap_M \Ann_R 0_{M}^{*_D(D+\Delta,\ba^t)}$, where $M$ runs through all $R$-modules.
\item $\Ann_R 0_{E}^{*_D(D+\Delta,\ba^t)}$, where $E=E_R(R/\m)$ is an injective hull of the residue field $R/\m$. 
\end{enumerate} 
\end{prop}
\begin{proof}
The equivalence of (1) and (3) follows from an argument very similar to that in the proof of \cite[Theorem 6.3]{Schw10}, which addresses the case $D=0$. 
The equivalence of (2) and (3) follows from an argument very similar to that in the proof of \cite[Proposition 8.23]{HH90}, which addresses the case $D=\Delta=0$ and $\ba=R$.   
\end{proof}

\begin{rem}\label{completion remark}
The formation of test ideals along divisors commutes with localization and completion (cf.~\cite[Corollary 3.6]{Tak08}). 
Therefore, by gluing, we can define test ideals along divisors for any $F$-finite normal schemes. 

Also, if $(R, \m)$ is local and $\widehat{R}$ denotes its $\m$-adic completion, then $E_R(R/\m)$ can be identified with $E_{\widehat{R}}(\widehat{R}/\m \widehat{R})$.
Under this identification, 
$0_{E_R(R/\m)}^{*_D(D+\Delta,\ba^t)}=
0_{E_{\widehat{R}}(\widehat{R}/\m \widehat{R})}^{*_{\widehat{D}}(\widehat{D}+\widehat{\Delta},(\ba \widehat{R})^t)}$, 
where $\widehat{D}$ and $\widehat{\Delta}$ are the flat pullbacks of $D$ and $\Delta$, respectively, via the natural map $\Spec \widehat{R} \to \Spec R$. 
Since the formation of test ideals along divisors commutes with completion, we have 
\begin{align*}
\Ann_{E_R(R/\m)} \tau_{D}(R,D+\Delta,\ba^t)
&=\Ann_{E_{\widehat{R}}(\widehat{R}/\m \widehat{R})} \tau_{D}(R,D+\Delta,\ba^t)\widehat{R}\\
&=\Ann_{E_{\widehat{R}}(\widehat{R}/\m \widehat{R})} \tau_{\widehat{D}}(\widehat{R}, \widehat{D}+\widehat{\Delta}, (\ba \widehat{R})^t)\\
&=0_{E_{\widehat{R}}(\widehat{R}/\m \widehat{R})}^{*_{\widehat{D}}(\widehat{D}+\widehat{\Delta},(\ba \widehat{R})^t)}\\
&=0_{E_R(R/\m)}^{*_D(D+\Delta,\ba^t)},
\end{align*}
where the third equality follows from \cite[Lemma 3.3]{Hara}. 
\end{rem}

\subsection{Absolute integral closure}
We recall the definition of absolute integral closure and a key result related to it that will be needed later. 
We refer the reader to \cite{Hu} for an overview of the theory of absolute integral closure.

Let $R$ be a normal domain with fractional field $K$ and $D$ be a prime divisor on $\Spec R$ with defining ideal $I_D=R(-D)$. 
We fix an algebraic closure $\overline{K}$ of $K$. 
The integral closure of $R$ in $\overline{K}$, denoted by $R^+$, is called an \textit{absolute integral closure} of $R$. 
Note that $R^+$ is independent, up to isomorphism, of the choice of $\overline{K}$.  
$(R/I_D)^+$ is defined in a similar way. 
We fix a prime ideal of $R^+$ lying over $I_D$ and write $I_D^+$ for it. 
This is an abuse of notation since $I_D^+$ is not uniquely determined by $D$. 
We have an isomorphism $R^+/I_D^+ \cong (R/I_D)^+$ (see \cite[p. 27]{Hoc07} for the proof). 

Assume that $R$ has a canonical module $\omega_R$, and let $S$ be a module-finite extension of a normal domain $R$ with $S$ normal and $\pi:\Spec S \to \Spec R$ denote a finite surjective morphism induced by the inclusion $R \hookrightarrow S$. 
The trace map $\mathrm{Tr}_{\pi}$ of $\pi$ is the map
\[
\omega_S \cong \Hom_R(S, \omega_R) \to \Hom_R(R, \omega_R)=\omega_R
\]
 induced by the inclusion $R \hookrightarrow S$. 
 
We use the following proposition in Section \ref{test submodule section}. 
\begin{prop}[cf.~\textup{\cite[Proposition 6.4]{MSTWW}}]\label{completion inclusion}
Suppose that $(R,\m)$ be an $F$-finite complete normal local ring of characteristic $p>0$. 
Let $D$ be a prime divisor and $\Delta$ be an effective $\Q$-Weil divisor on $X:=\Spec R$ such that $K_X+D+\Delta$ is $\Q$-Cartier and no component of $\Delta$ is equal to $D$. 
We fix a choice of $I_{D}^+$. 
For every module-finite extension $S$ of $R$ contained in $R^+$ with $S$ normal, 
the trace map $\mathrm{Tr}_{\pi}$ induces a map 
\[\pi_*\omega_S(D_S-\lfloor \pi^*(K_X+D+\Delta) \rfloor) \to R\]
and the ideal $\tau_D(R, D+\Delta)$ is contained in its image,  
where $\pi:\Spec S \to \Spec R$ is a finite surjective morphism induced by the inclusion $R \hookrightarrow S$ and $D_S$ is the prime divisor on $\Spec S$ such that $I_D^+ \cap S=S(-D_S)$. 
\end{prop}

\begin{proof}
The image of $\mathrm{Tr}_{\pi}:\pi_*\omega_S(D_S-\lfloor \pi^*(K_X+D+\Delta) \rfloor) \to R$ is independent of the choice of $K_X$, so we may assume that $K_X$ is effective. 
By assumption, $r(K_X+D+\Delta)=\Div_R(f)$ for some integer $r \ge 1$ and a  nonzero element $f \in R$. 
Let $S \hookrightarrow T$ be a module-finite extension of $S$ contained in $R^+$ such that $T$ is a normal domain and $f^{1/r} \in T$, and let $\psi: \Spec T \to \Spec R$ denote the morphism corresponding to the inclusion $R \hookrightarrow T$. 
Since the image of 
\[\mathrm{Tr}_{\psi}: \psi_*\omega_T(D_T - \lfloor \psi^*(K_X+D+\Delta)  \rfloor) \to R,\]
where $D_T$ is the prime divisor on $\Spec T$ such that $I_D^+ \cap T=T(-D_T)$, is contained in 
\[\mathcal{I}_{S, D}(R, D+\Delta):=\mathrm{Im}\left(\mathrm{Tr}_{\pi}: \pi_*\omega_S(D_S - \lfloor \pi^*(K_X+D+\Delta) \rfloor) \to R\right),\]
we may assume that $f^{1/r} \in S$. 
Observing that 
\[\pi^*K_X-\Div_S(f^{1/r})=\pi^*K_X-\pi^*(K_X+D+\Delta) \le -D_S,\] 
we see that the multiplication by $f^{1/r}$ induces an $R$-module homomorphism $\omega_R=\mathcal{O}_X(K_X) \xrightarrow{\cdot f^{1/r}} S(-D_S)$. 
It then follows from \cite[Proposition 6.4]{MSTWW} that $\tau_D(R, D+\Delta)$ is contained in the BCM adjoint ideal $\tau_{+, D}(R, D+\Delta),$\footnote{The ideal $\tau_{+, D}(R, D+\Delta)$ is denoted by $\adj^D_{R^+ \to (R/I_D)^+}(R, D+\Delta)$ in the notation of \cite{MSTWW}.} which is defined as 
\[
\tau_{+, D}(R, D+\Delta)=\Ann_R \ker \left(H^d_{\m}(\omega_R) \xrightarrow{\cdot f^{1/r}} H^d_{\m}(S(-D_S)) \to H^d_{\m}(I_D^+)\right).
\]
Note that by \cite[Lemma 2.5 (b)]{MSTWW}, the formation of $\tau_{+,D}(R, D+\Delta)$ is independent of the choice of $K_X$. 
Therefore, it suffices to show that $\tau_{+, D}(R, D+\Delta) \subseteq \mathcal{I}_{S, D}(R, D+\Delta)$. 
However, this is immediate because 
\[\mathcal{I}_{S, D}(R, D+\Delta)=\Ann_R \ker \left(H^d_{\m}(\omega_R) \xrightarrow{\cdot f^{1/r}} H^d_{\m}(S(-D_S))\right)
\]
by Matlis duality. 
\end{proof}

\subsection{Adjoint ideal sheaves}
In this subsection, we define multiplier ideal sheaves and adjoint ideal sheaves in the non-$\Q$-Gorenstein setting. 
Our main reference is \cite{dFH}, and we use the notation in \cite{KM} and \cite{Laz} freely. 

Let $X$ be a normal quasi-projective variety over an algebraically closed field $k$ of characteristic zero, $D$ be a reduced divisor on $X$ and $\Delta$ be an effective $\R$-Weil divisor on $X$ that has no common components with $D$. 
Let $t \ge 0$ be a real number and $\ba \subseteq \sO_X$ be a coherent ideal sheaf such that no components of $D$ are contained in the zero locus of $\ba$. 

\begin{defn}\label{adjoint def}
\begin{enumerate}
\item Suppose that $K_X+D+\Delta$ is $\R$-Cartier, and take a log resolution $\pi:\widetilde{X} \to X$ of $(X, D+\Delta, \ba)$ such that $\ba \sO_{\widetilde{X}}= \sO_{\widetilde{X}}(-F)$ for an effective divisor $F$ on $\widetilde{X}$ and the strict transform $\pi^{-1}_*D$ of $D$ is smooth (but possibly disconnected). 
Then the \textit{adjoint ideal sheaf} $\mathrm{adj}_D(X, D+\Delta, \ba^t)$ of the triple $(X, D+\Delta, \ba^t)$ along $D$ is defined as 
\[
\mathrm{adj}_D(X, D+\Delta, \ba^t)=\pi_*\sO_{\widetilde{X}}(K_{\widetilde{X}}-\lfloor \pi^*(K_X+D+\Delta)+t F \rfloor+\pi^{-1}_* D). 
\]
The definition is independent of the choice of $\pi$. 
When $\ba=\sO_X$, the ideal sheaf $\mathrm{adj}_D(X, D+\Delta, \ba^t)$ is simply denoted by $\mathrm{adj}_D(X, D+\Delta)$. 
 
\item 
If $K_X+D+\Delta$ is not $\R$-Cartier, then 
the \textit{adjoint ideal sheaf} $\mathrm{adj}_D(X, D+\Delta, \ba^t)$ of the triple $(X, D+\Delta, \ba^t)$ along $D$ is defined as 
\[
\mathrm{adj}_D(X, D+\Delta, \ba^t)=\sum_{\Delta'} \mathrm{adj}_D(X, D+\Delta+\Delta', \ba^t),
\]
where $\Delta'$ runs through all effective $\R$-Weil divisors on $X$ such that $D$ and $\Delta'$ have no common components and $K_X+D+\Delta+\Delta'$ is $\R$-Cartier. 
When $D=0$, it is denoted by $\J(X, \Delta, \ba^t)$ and called the \textit{multiplier ideal sheaf} of the triple $(X, \Delta, \ba^t)$. 
When $\ba=\sO_X$, the ideal sheaf $\mathrm{adj}_D(X, D+\Delta, \ba^t)$ (resp.~$\J(X, \Delta, \ba^t)$) is simply denoted by $\mathrm{adj}_D(X, D+\Delta)$ (resp.~$\J(X, \Delta)$). 

\item 
The pair $(X, \Delta)$ is said to be of \textit{klt type} (resp.~\textit{lc type}) if there exists an effective $\R$-Weil divisor $\Delta'$ on $X$ such that $K_X+\Delta+\Delta'$ is $\R$-Cartier and the pair $(X, \Delta+\Delta')$ is klt (resp. lc).
\end{enumerate}
\end{defn}

\begin{rem}
$(X, \Delta)$ is of klt type if and only if $\J(X, \Delta)=\sO_X$. The pair $(X, D+\Delta)$ is of plt type along $D$ if and only if $\mathrm{adj}_D(X, D+\Delta)=\sO_X$ (see the paragraph following Theorem \ref{Zhuang thm} for the definition of being of plt type). 
\end{rem}


\begin{prop}\label{single divisor}
There exists an effective $\R$-Weil divisor $\Gamma$ on $X$ such that $D$ and $\Gamma$ have no common components, $K_X+D+\Delta+\Gamma$ is $\Q$-Cartier and  
\[
 \mathrm{adj}_D(X, D+\Delta, \ba^t)=\mathrm{adj}_D(X, D+\Delta+\Gamma, \ba^t).
 \]
 In particular, if $(X, D+\Delta)$ is of plt type along $D$, then 
 $(X, D+\Delta+\Gamma)$ is plt along $D$. 
\end{prop}
\begin{proof}
In Definition \ref{adjoint def} (2), $\Delta'$ runs through all effective $\R$-Weil divisors on $X$ such that $D$ and $\Delta'$ have no common components and $K_X+D+\Delta+\Delta'$ is $\R$-Cartier. 
First, we will show that it suffices to consider such $\Delta'$ with $K_X+D+\Delta+\Delta'$ $\Q$-Cartier. 
Suppose $\Delta'$ is an $\R$-Weil divisor on $X$ such that $D$ and $\Delta'$ have no common components and $K_X+D+\Delta+\Delta'$ is $\R$-Cartier. 
Write $K_X+D+\Delta+\Delta'=\sum_{i= 1}^{n} a_iC_i$, where $n \ge 1$ is an integer, $C_i$ is a Cartier divisor on $X$ for each $i=1,\dots,n$ and $a_1,\dots,a_n$ is a sequence of real numbers linearly independent over $\Q$. 
Since $\sum_{i=1}^n a_i \ord_D(C_i)=0$, we have $\ord_D(C_i)=0$, that is, each $C_i$ has no common components with $D$. 
Take an effective Cartier divisor $A$ on $X$ such that $A$ has no common components with $D$ and $A+C_i$ is effective for all $i$. 
Let $B=(n+1)A+\sum_{i=1}^{n}C_i$, and choose a sufficiently small real number $\epsilon>0$ so that
\[
    \adj_D(X,D+\Delta+\Delta'+\epsilon B,\ba^t)
    =\adj_D(X,D+\Delta+\Delta',\ba^t).
\]
Pick a $\Q$-Cartier $\Q$-Weil divisor $\Theta$ on $X$ such that 
\[K_X+D+\Delta+\Delta'\le \Theta\le K_X+D+\Delta+\Delta'+\epsilon B,\]
and let $\Delta'':=\Theta-(K_X+D+\Delta)$. 
Then $\Delta''$ has no common components with $D$, $K_X+D+\Delta+\Delta''$ is $\Q$-Cartier and $\adj_D(X,D+\Delta+\Delta'',\ba^t)=\adj_D(X,D+\Delta+\Delta',\ba^t)$. 
Thus, we may assume that $K_X+D+\Delta+\Delta'$ is $\Q$-Cartier in Definition \ref{adjoint def} (2). 

Next, we reduce to the case where $\Delta$ is a $\Q$-Weil divisor. 
By the argument above, there exist finitely many $\R$-Weil divisors $\Delta'_1, \dots, \Delta'_r$ on $X$ such that $D$ and $\Delta'_i$ have no common components, $K_X+D+\Delta+\Delta'_i$ is $\Q$-Cartier for all $i=1, \dots, r$ and 
\[
\adj_D(X,D+\Delta, ,\ba^t)=\sum_{i=1}^r \adj_D(X,D+\Delta+\Delta'_i,\ba^t).
\]
Let $\widetilde{\Delta}$ be an $\Q$-Weil divisor on $X$ such that $\Delta+\Delta'_i \ge \widetilde{\Delta} \ge \Delta$ for all $i=1, \dots, r$. 
Then 
\[
\sum_{i=1}^r \adj_D(X,D+\Delta+\Delta'_i,\ba^t) \subseteq \adj_D(X,D+\widetilde{\Delta},\ba^t) \subseteq \adj_D(X,D+\Delta,\ba^t), 
\]
which implies that $\adj_D(X,D+\widetilde{\Delta},\ba^t)=\adj_D(X,D+\Delta,\ba^t)$. 
Therefore, replacing $\Delta$ with $\widetilde{\Delta}$, we may assume that $\Delta$ is a $\Q$-Weil divisor. 

For every integer $m \ge 2$ such that $m(K_X+D+\Delta)$ is an integral Weil divisor, take a log resolution $\pi:Y \to X$ of $(X, \sO_X(-m(K_X+D+\Delta)) \ba)$ such that $\ba \sO_Y=\sO_Y(-F)$ and $\sO_X(-m(K_X+D+\Delta))\sO_Y=\sO_Y(-G_m)$ for effective divisors $F$ and $G_m$ on $Y$. 
Then we define an ideal sheaf $\mathrm{adj}_D^{(m)}(X, D+\Delta, \ba^t)$ as 
\[
\mathrm{adj}_D^{(m)}(X, D+\Delta, \ba^t)=\sO_Y\left(K_Y- \left\lfloor \frac{G_m}{m}+tF \right\rfloor +\pi^{-1}_*D\right). 
\]
For every effective $\Q$-Weil divisor $\Delta'$ on $X$ such that $D$ and $\Delta'$ have no common components and $m(K_X+D+\Delta+\Delta')$ is Cartier, we have $G_m \le m \pi^*(K_X+D+\Delta+\Delta')$ and consequently, 
\[\mathrm{adj}_D(X, D+\Delta+\Delta', \ba^t) \subseteq \mathrm{adj}_D^{(m)}(X, D+\Delta, \ba^t).\]
The family $\{\mathrm{adj}_D^{(m)}(X, D+\Delta, \ba^t)\}_m$ of ideal sheaves has a unique maximal element, which is denoted by $\mathrm{adj}'_D(X, D+\Delta, \ba^t)$. 
By the above observation, we have an inclusion $\mathrm{adj}_D(X, D+\Delta, \ba^t) \subseteq \mathrm{adj}'_D(X, D+\Delta, \ba^t)$. 
On the other hand, it follows from an argument similar to the proof of \cite[Proposition 5.4]{dFH} that there exists an effective $\Q$-Weil divisor $\Gamma$ on $X$ such that $D$ and $\Gamma$ have no common components, $K_X+D+\Delta+\Gamma$ is $\Q$-Cartier and  
\[
\mathrm{adj}'_D(X, D+\Delta, \ba^t)=\mathrm{adj}_D(X, D+\Delta+\Gamma, \ba^t) \subseteq \mathrm{adj}_D(X, D+\Delta, \ba^t). 
\]
Thus, we conclude that $\mathrm{adj}_D(X, D+\Delta, \ba^t)=\mathrm{adj}_D(X, D+\Delta+\Gamma, \ba^t)$. 
\end{proof}

\begin{rem}
Unlike Proposition \ref{single divisor}, when $(X, \Delta)$ is of lc type, there does not generally exist an effective $\R$-Weil divisor $\Gamma$ on $X$ such that $K_X+\Delta+\Gamma$ is $\Q$-Cartier and $(X, \Delta+\Gamma)$ is lc. 
\end{rem}

\subsection{Ultraproducts}
In this subsection, we quickly review the theory of ultraproducts in commutative algebra. We refer the reader to \cite{Sch03} and \cite[Section 3]{Yam} for details.
Throughout this subsection, $\mathcal{P}$ denotes the set of prime numbers.
\begin{defn}
	A non-empty subset $\mathcal{F}$ of the power set of $\mathcal{P}$ is said to be a {\it non-principal ultrafilter} if the following four conditions hold.
	\begin{enumerate}
		\item If $A, B \in \mathcal{F}$, then $A\cap B\in \mathcal{F}$.
		\item If $A\in \mathcal{F}$ and $A\subseteq B \subseteq \mathcal{P}$, then $B\in \mathcal{F}$.
		\item For any $A\subseteq \mathcal{P}$, we have $A\in \mathcal{F}$ or $\mathcal{P}\setminus A \in \mathcal{F}$.
		\item For any finite subset $A\subseteq \mathcal{P}$, we have $A\notin \mathcal{F}$.
	\end{enumerate}
\end{defn}

From now on, we fix a non-principal ultrafilter $\mathcal{F}$ on $\mathcal{P}$. 

\begin{defn}
Let $\phi$ be a property on $\mathcal{P}$. 
We say that $\phi(p)$ holds {\it for almost all $p$} if the set $\{p\in\mathcal{P}\mid\text{$\phi(p)$ holds}\}$ belongs to $\mathcal{F}$.
\end{defn}
\begin{defn}
Let $(A_p)_{p \in \mathcal{P}}$ be a family of non-empty sets indexed by $\mathcal{P}$. 
Its {\it ultraproduct} is defined as
\[
A_\infty =\ulim_p A_p :=\prod_{p}A_p/\sim,
\]
where $(a_p)_{p \in \mathcal{P}} \sim (b_p)_{p \in \mathcal{P}}$ if and only if  $a_p=b_p$ for almost all $p$.
If $(A_p)_{p \in \mathcal{P}}$ is a family of rings (resp.~fields), then $A_\infty$ is a ring (resp.~a field). 
\end{defn}

Let $\overline{\F_p}$ be an algebraic closure of the prime field $\F_p$ of characteristic $p>0$ and then there exists a non-canonical isomorphism 
\[
\C \cong \ulim_p \overline{\F_p}
\]
of fields (see, for example, \cite[Theorem 2.4]{Sch03}). We fix this isomorphism, and let $R$ be a local ring essentially of finite type over $\C$. 
An \textit{approximation} of $R$ and the \textit{non-standard hull} of $R$ are constructed via the fixed isomorphism (see \cite[Subsection 4.3]{Sch03} and \cite[Definition 3.35]{Yam}). 
We omit formal definitions here and provide instead a brief overview of their fundamental properties.
 \begin{enumerate}
 \item An approximation of $R$ is a family $(R_p)_{p \in \mathcal{P}}$ of rings $R_p$ essentially of finite type over $\bar{\F_p}$ indexed by $\mathcal{P}$ and $R_p$ is a local ring for almost all $p$. If $R$ is normal (resp.~regular), then so is $R_p$ for almost all $p$ (see \cite[Theorem 4.6]{Sch03} and \cite[Proposition 3.9]{Yam22a}).
\item An approximation of an element $x \in R$ is a sequence $(x_p)_{p \in \mathcal{P}}$ of elements $x_p \in R_p$ indexed by $\mathcal{P}$. 
 \item The non-standard hull $R_\infty$ is the ultraproduct of an approximation $(R_p)_{p \in \mathcal{P}}$ of $R$. 
 Once we have fixed the ultrafilter $\mathcal{F}$ and the isomorphism $\C \cong \ulim_p \overline{\F_p}$, it becomes independent of the choice of approximations. 
 We also remark that $R_\infty$ is not Noetherian in general. 
 \item (\cite[Corollary 4.3]{Sch03}) The map $R \to  R_\infty$ sending $x$ to the image of $(x_p)$ is a faithfully flat injection. 
 \end{enumerate}

\begin{eg}
Let $f_1,\dots,f_r\in (x_1,\dots,x_n)\Z[x_1,\dots,x_n]$ be polynomials in the variables $x_1, \dots, x_n$ with integer coefficients.
Set  
\[
R=\left(\C[x_1,\dots,x_n]/(f_1,\dots,f_r)\right)_{(x_1,\dots,x_n)}.
\] 
For each $p \in \mathcal{P}$, let $R_p$ denote the base change of its reduction modulo $p$ to an algebraic closure $\overline{\F_p}$ of the prime field $\F_p$, that is, 
\[R_p=(\overline{\F_p}[x_1,\dots,x_n]/(\overline{f_1},\dots,\overline{f_r}))_{(x_1,\dots,x_n)},\] where $\overline{f_1}, \dots, \overline{f_r}$ are the image of $f_1, \dots, f_r$ in $\overline{\F_p}[x_1,\dots,x_n]$.
Then the family $(R_p)_{p\in \mathcal{P}}$ is an approximation of $R$.
\end{eg}

If $\phi:R\to R'$ is a local $\C$-algebra homomorphism between local rings essentially of finite type over $\C$, then we can define an approximation $(\phi_p:R_p\to R'_p)_{p\in \mathcal{P}}$ of $\phi$ as in \cite[3.2.4]{Sch03} and \cite[Definition 3.40]{Yam}, which is a family of local $\bar{\F_p}$-algebra homomorphisms such that the following diagram commutes:
\[
\xymatrix{
R \ar[rr]^{\phi} \ar[d] & & R' \ar[d] \\
R_\infty \ar[rr]^-{\ulim_p \phi_p} & & R'_\infty
	}.
\]

Let $S$ be a finitely generated $R$-algebra. An \textit{$R$-approximation} of $S$ and the (relative) \textit{$R$-hull} of $S$ are constructed similarly. 
The reader is referred to \cite[Section 2]{Sch08} and \cite[Subsection 3.3]{Yam} for the details. 
Here is a summary of some basic properties of the $R$-hull.
\begin{enumerate}
\item An $R$-approximation of $S$ is a family $(S_p)_{p \in \mathcal{P}}$ of finitely generated $R_p$-algebras $S_p$ indexed by $\mathcal{P}$. 
\item The relative $R$-hull $S_\infty$ of $S$ is the ultraproduct of an $R$-approximation $(S_p)_{p \in \mathcal{P}}$. 
It is independent of the choice of approximations. 
\item We have a natural faithfully flat inclusion $S\hookrightarrow S_\infty$.
\end{enumerate}

The above construction can be generalized in several ways. 
Given an $S$-module $N$, an approximation 
of $N$ is defined as in \cite[Subsection 2.5]{Sch05}, \cite[Definition 3.41]{Yam} and \cite[Definition 3.60]{Yam}. 
If $R$ is normal and $\Delta$ is an $\R$-Weil divisor on $\Spec R$, then an approximation $(\Delta_p)_{p \in \mathcal{P}}$ of $\Delta$ is defined as in \cite[Definition 3.45]{Yam}, where $\Delta_p$ is an $\R$-Weil divisor on $\Spec R_p$ for almost all $p$. 
Moreover, if $\Delta$ is $\Q$-Cartier of index $r$, then so is $\Delta_p$ for almost all $p$. 
We can also define an approximation $(X_p)_{p \in \mathcal{P}}$ of an algebraic variety $X$ over $\C$, where  $X_p$ is an algebraic variety over $\overline{\F_p}$ for almost all $p$. 
The reader is referred to \cite{Sch05} and \cite[Subsection 3.3]{Yam} for the details.

Following Schoutens \cite[Section 5]{Sch08}, we can define local ultracohomology modules. 
Let $x_1,\dots,x_d$ be a system of parameters for $R$. Suppose that $M_p$ is an $R_p$-module for almost all $p$ and $M_\infty=\ulim_p M_p$. For an integer $n \ge 1$ and an $n$-tuple $1\le i_1<\dots<i_n\le d$, there exists a natural map
\[
(M_\infty)_{x_{i_1}\cdots x_{i_n}}\to \ulim_p (M_p)_{x_{i_1,p}\cdots x_{i_n,p}}.
\]
Considering the \v{C}ech complexes of $M_p$ and $M_{\infty}$, we have commutative diagrams
\[
\xymatrix{
	\bigoplus_{1\le i_1<\cdots<i_n \le d}(M_\infty)_{x_{i_1}\dots x_{i_n}} \ar[r] \ar[d]& \bigoplus_{1\le j_1<\cdots<j_{n+1}\le d} (M_\infty)_{x_{j_1}\dots x_{j_{n+1}}} \ar[d]\\
	\bigoplus_{1\le i_1<\cdots<i_n \le d}\ulim_p (M_p)_{x_{i_1,p}\dots x_{i_n,p}} \ar[r] & \bigoplus_{1\le j_1<\cdots<j_{n+1}\le d}\ulim_p (M_p)_{x_{j_1,p}\dots x_{j_{n+1},p}}, 
}
\]
which yield natural maps 
\[
H_{\m}^i(M_\infty)\to \ulim_p H_{\m_p}^i(M_p).
\]
We do not know whether this map is injective or not in general, and revisit this problem in Section 5.

\begin{lem}\label{lemma tensor ultraproduct}
Suppose that $S$ is a module-finite extension of $R$ contained in $R^+$. 
Let $(S_p)_{p \in \mathcal{P}}$ be an $R$-approximation of $S$, $(M_p)_{p \in \mathcal{P}}$ be a family of $S_p$-modules indexed by $\mathcal{P}$ and $N$ be a finite $S$-module.
Then 
\[
(\ulim_p M_p)\otimes_S N \cong \ulim_p (M_p\otimes_{S_p}N_p).
\]
\end{lem}

\begin{proof}
Take a finite presentation 
\[
	S^m \xrightarrow{A} S^n \to N \to 0
\]
of the $S$-module $N$, where $m$, $n$ are positive integers and $A$ is an $n \times m$ matrix with entries in the maximal ideal $\m$. 
Then we have an exact sequence
\[
S_p^m \xrightarrow{A_p} S_p^n \to N_p \to 0
\]
for almost all $p$, where $N_p$ and $A_p$ are approximations of $N$ and $A$, respectively. 
Tensoring with $M_p$ yields the exact sequence 
\[
M_p^m \xrightarrow{A_p} M_p^n \to M_p\otimes_{S_p}N_p \to 0
\]
for almost all $p$. 
Taking its ultraproduct, we have an exact sequence
\[
(\ulim_p M_p)^m \xrightarrow{A} (\ulim_p M_p)^n \to \ulim_p (M_p\otimes_{S_p}N_p) \to 0, 
\]
which induces the isomorphism 
	\[
	(\ulim_p M_p)\otimes_S N \cong \ulim_p (M_p\otimes_{S_p}N_p).
	\]
\end{proof}

We conclude this subsection by discussing properties inherited by approximations. 

\begin{prop}[{\cite[Theorem 4.7]{Sch03}}]\label{reduction modulo p injectivity}
Let $\iota: R \hookrightarrow S$ be an injective local $\C$-algebra homomorphism between local domains essentially of finite type over $\C$. 
If $(\iota_p:R_p \to S_p)_{p \in \mathcal{P}}$ is an approximation of $\iota$, then $\iota_p$ is injective for almost all $p$.
\end{prop}
\begin{rem}
Flatness is preserved under taking approximations (cf.~Proposition \ref{proposition redutions mod p>0 flat local}), but purity is not (see \cite{HJPS}).
\end{rem}

\section{Test submodules along divisors}\label{test submodule section}

In this section, we develop the theory of test submodules along divisors. Throughout this section, we work with the following setting. 

\begin{setting}\label{setting test submodules}
Let $R$ be a $d$-dimensional $F$-finite normal domain of characteristic $p>0$, $D$ be a reduced divisor on $X:=\Spec R$ and $\Gamma$ be an effective $\Q$-Weil divisor on $X$ that has no common components with $D$. 
We assume that $F^!\omega_X^{\bullet} \cong \omega_X^{\bullet}$, where $F: X \to X$ is the Frobenius morphism and $\omega_X^{\bullet}$ is a  normalized dualizing complex for $X$. 
This condition is satisfied, for example, when $R$ is essentially of finite type over an $F$-finite local ring (see \cite[Example 2.15]{BST}). 
Note that $h^{-d}(\omega_X^{\bullet})$ is isomorphic to the canonical module $\omega_R$ of $R$. 
\end{setting}

\begin{defn}\label{parameter test defn}
Let the notation be as in Setting \ref{setting test submodules}. 
The \textit{parameter test submodule} $\tau_D(\omega_R, \Gamma)$ of the pair $(R, \Gamma)$ along $D$ is defined as the unique smallest submodule $M$ of $\omega_R(D)$ satisfying the following conditions:  
\begin{enumerate}[label=(\alph*)]
\item $M$ coincides with $\omega_R(D)$ at every generic point of $D$. 
\item For every integer $e \ge 0$ and every $\phi \in \Hom_R(F^e_*\omega_R(p^eD+\lceil (p^e-1)\Gamma \rceil), \omega_R(D)) \subseteq \Hom_R(F^e_*\omega_R(D), \omega_R(D))$, one has $\phi(F^e_*M) \subseteq M$. 
\end{enumerate}
\end{defn}

\begin{rem}\label{parameter submodule with ideal}
Given an ideal $\ba \subseteq R$ such that $\ba \cap R^{\circ, D} \ne \emptyset$ and a real number $t \ge 0$, we can define the parameter test submodule $\tau_D(\omega, \ba^t)$ of the pair $(R, \ba^t)$ along $D$ similarly. This is the unique smallest submodule $M$ of $\omega_R(D)$ satisfying the following conditions:
\begin{enumerate}[label=(\alph*)]
\item $M$ coincides with $\omega_R(D)$ at every generic point of $D$. 
\item[(b')] For every integer $e \ge 0$ and every $\phi \in \Hom_R(F^e_*\omega_R(p^eD), \omega_R(D))$,  which is viewed as an element of $\Hom_R(F^e_*\omega_R(D), \omega_R(D))$, one has $\phi(F^e_* \ba^{\lceil t(p^e-1) \rceil}M) \subseteq M$. 
\end{enumerate}
\end{rem}

\begin{lem}\label{test ideal interpretation}
With notation as in Setting \ref{setting test submodules}, choose a canonical divisor $K_X$ of $X:=\Spec R$ such that $-(K_X+D)$ is an effective Weil divisor $G$ with no common components with $D$, and fix $\omega_R=R(K_X)$ to be the corresponding fractional ideal of $R$. 
Then we have the equality 
\[
\tau_D(\omega_R, \Gamma)=\tau_D(R, D+\Gamma+G)
\]
of fractional ideals of $R$. In particular, $\tau_D(\omega_R, \Gamma)$ exists. 
\end{lem}

\begin{proof}
First, we verify that $\tau_D(R,D+\Gamma+G)$ is a submodule of $\omega_R(D)$. 
Noting that for all powers $q=p^e$ of $p$, we have the isomorphism 
\[
    \Hom_R(F^e_*R(\lceil (q-1)(D+\Gamma+G) \rceil), R)\cong \Hom_R(F^e_*R(\lceil (q-1)(D+\Gamma)-G \rceil), R(-G))
\]
and the inclusion $F^e_*R(-G) \subseteq F^e_*R(\lceil (q-1)(D+\Gamma)-G \rceil)$, we see that $\varphi(F^e_*R(-G))$ is a submodule of $R(-G)$ for all $\varphi \in \Hom_R(F^e_*R(\lceil (q-1)(D+\Gamma+G) \rceil), R)$. 
Moreover, $R(-G)\cap R^{\circ,D} \neq \emptyset$, as $G$ has no common components with $D$.
Thus, by the minimality of the test ideal along $D$, it follows that $\tau_D(R,D+\Gamma+G)\subseteq R(-G) =\omega_R(D)$.

Next, we show that $\tau_D(R,D+\Gamma+G)=\tau_D(\omega_R, \Gamma)$. 
Since 
\[
\Hom_R(F^e_*R(\lceil (q-1)(D+\Gamma+G)\rceil), R) \cong \Hom_R(F^e_*\omega_R(qD+\lceil (q-1)\Gamma \rceil), \omega_R(D))
\]
for all powers $q=p^e$, the ideal $\tau_D(R, D+\Gamma+G)$ is the smallest submodule $M$ of $\omega_R(D)$ with $M \cap R^{\circ, D} \ne \emptyset$ satisfying the condition (b) in Definition \ref{parameter test defn}. 
On the other hand, a submodule $N$ of $\omega_R(D)$ satisfies that $N \cap R^{\circ, D} \ne \emptyset$ if and only if there exists an element $c \in R^{\circ, D}$ such that $c \omega_R(D) \subseteq N$, which is equivalent to the condition (a) in Definition \ref{parameter test defn}. 
Therefore, $\tau_D(R, D+\Gamma+G)$ coincides with $\tau_D(\omega_R, \Gamma)$. 
\end{proof}

\begin{rem}\label{basic properties of parameter test}
Thanks to Lemma \ref{test ideal interpretation}, several basic properties of $\tau_D(\omega_R, \Gamma)$ can be deduced from the corresponding properties of $\tau_D(R, \Delta)$. For example, 
\begin{enumerate}
\item the formation of $\tau_D(\omega_R, \Gamma)$ commutes with localization, 
\item if $(R,\m)$ is local, then the formation of $\tau_D(\omega_R, \Gamma)$ commutes with $\m$-adic completion, and 
\item if $B$ is an effective Cartier divisor on $\Spec R$ with no common components with $D$, then $\tau_D(\omega_R, \Gamma+B)=\tau_D(\omega_R, \Gamma) \otimes_R R(-B)$. 
\end{enumerate}
\end{rem}

For each integer $e \ge 1$, let 
\[
R \to F^e_*R \hookrightarrow F^e_*R(\lceil (p^e-1) (D+\Gamma) \rceil)  \tag{$\star$}
\]
be the composite of the $e$-times iterated Frobenius map $R \to F^e_*R$ and the pushforward of the natural inclusion $R \hookrightarrow R(\lceil (p^e-1) (D+\Gamma) \rceil)$ by $F^e$. 

\begin{defn}\label{test submodule local}
With notation as in Setting \ref{setting test submodules}, suppose that $(R,\m)$ is local. 
Tensoring $(\star)$ with $I_D:=R(-D)$ and taking local cohomology, one has a map 
\[
F^e_{D, \Gamma}: H^d_{\m}(I_D) \to H^d_{\m} (I_D(\lceil (p^e-1) \Gamma \rceil)), 
\]
where $I_D(\lceil (p^e-1) \Gamma \rceil):=R(-D+\lceil (p^e-1) \Gamma \rceil)$.  
The submodule 
$Z(D;\Gamma)$
of $H^d_\m(I_D)$ consists of all elements $z \in H^d_\m(I_D)$ for which there exists an element $c \in R^{\circ, D}$ such that 
\[c  F^e_{D, \Gamma}(z)=0 \in H^d_{\m} (I_D(\lceil (p^e-1) \Gamma \rceil))\]
for all large $e$. 
It is simply denoted as $Z(D)$ when $\Gamma=0$. 
\end{defn}

\begin{lem}\label{test submodule dual}
With notation as in Setting \ref{test submodule local}, 
let $E_R(R/\m)$ be an injective hull of the residue field $R/\m$. Then 
\[
\tau_{D}(\omega_R, \Gamma)=\Ann_{\omega_R(D)} Z(D;\Gamma), 
\]
the annihilator of $Z(D;\Gamma)$ in $\omega_R(D)$ with respect to the duality pairing 
\[\omega_R(D) \times H^d_{\m}(I_D) \to E_R(R/\m).\]
\end{lem}
\begin{proof}
We choose a canonical divisor $K_X$ of $X:=\Spec R$ such that $G:=-(K_X+D)$ is an effective Weil divisor with no common components with $D$. 
By Lemma \ref{test ideal interpretation}, we have $\tau_D(\omega_R,\Gamma)=\tau_D(R,D+\Gamma+G)$. 
Let $c'\in R(-\lceil \Gamma \rceil)\cap R^{\circ,D}$. 
Since 
\[
    H_\m^d(\omega_R)\otimes_R F^e_*R((p^e-1)D+\lceil p^e(\Gamma+G) \rceil)\cong F^e_*H_\m^d(I_D(\lceil p^e\Gamma \rceil)),
\]
for each $c\in R^{\circ,D}$ and $e\in \N$, 
we obtain the following commutative diagram
\[
    \xymatrix{
    H_\m^d(\omega_R) \ar[r]^-{\cdot F^e_*c} \ar[d] & F^e_*H_\m^d(I_D(\lceil p^e\Gamma\rceil)) \ar[d]^-{\cdot F^e_*c'} \\
    H_\m^d(I_D) \ar[r]^-{\cdot F^e_*(cc')} \ar[rd]
    & F^e_*H_\m^d((I_D(\lceil (p^e-1)\Gamma\rceil)) \ar[d]\\
        & F^e_*H_\m^d(I_D(\lceil p^e\Gamma\rceil)), 
    }
\]
which yields another commutative diagram:
\[
    \xymatrix{
         0_{H_\m^d(\omega_R)}^{*_D(D+\Gamma+G)} \ar@{^{(}->}[r] \ar@{->>}[d] & H_\m^d(\omega_R) \ar@{->>}[d]\\
        Z(D;\Gamma) \ar@{^{(}->}[r] & H_\m^d(I_D).
    }
\]
Let $\tau'_D(\omega_R,\Gamma)$ denote the right hand side of the desired formula. 
The Matlis dual of the above diagram induces a natural inclusion 
\[\omega_R(D)/\tau'_D(\omega_R,\Gamma) \hookrightarrow R/\tau_D(R,D+\Gamma+G).\] 
Here, note that $\omega_R(D)$ is an ideal of $R$ by our choice of $K_X$. 
This injectivity implies that 
\[\tau_D(R,D+\Gamma+G) \cap \omega_R(D)=\tau'_D(\omega_R, \Gamma).\] 
Moreover, as shown in the proof of Lemma \ref{test ideal interpretation}, $\tau_D(R,D+\Gamma+G)$ is a submodule of $\omega_R(D)$. 
Thus, we conclude that 
\[
\tau_D(\omega_R,\Gamma)=\tau_D(R,D+\Gamma+G)=\tau'_D(\omega_R,\Gamma).
\]
\end{proof}

In Proposition \ref{parameter test=image of trace} and Lemma \ref{equational lemma},  
we assume that $D$ is a prime divisor for simplicity. 
We then fix a choice of $I_{D}^+$ and use the following notation.  
Given a module-finite extension $S$ of $R$ contained in $R^+$ with $S$ normal, we define the submodule $\mathcal{I}_S(\omega_R(D), \Gamma)$ of $\omega_R(D)$ as 
\[\mathcal{I}_S(\omega_R(D), \Gamma)=\mathrm{Im}(\mathrm{Tr}_{\pi}: \pi_*\omega_S(D_S - \lfloor \pi^* \Gamma \rfloor) \to \omega_R(D)), \]
where $\pi:\Spec S \to \Spec R$ is the finite surjective morphism induced by the inclusion $R \hookrightarrow S$ and $D_S$ is the prime divisor on $\Spec S$ such that $I_D^+ \cap S=S(-D_S)$. 
When $\Gamma=0$, this submodule is simply denoted by $\mathcal{I}_S(\omega_R(D))$. 

\begin{prop}\label{parameter test=image of trace}
With notation as above and as in Setting \ref{setting test submodules}, 
suppose that $D$ is a prime divisor and $\Gamma$ is $\Q$-Cartier. 
\begin{enumerate}
\item 
For every module-finite extension $S$ of $R$ contained in $R^+$ with $S$ normal, one has 
\[
\tau_{D}(\omega_R, \Gamma) \subseteq \mathcal{I}_S(\omega_R(D), \Gamma).
\]
\item 
There exists a module-finite extension $S$ of $R$ contained in $R^+$ such that $S$ is normal, $\pi^*\Gamma$ is Cartier, and the equality holds in (1), that is, 
\[
\tau_{D}(\omega_R, \Gamma)=\mathcal{I}_S(\omega_R(D), \Gamma). 
\]
\end{enumerate}
\end{prop}
\begin{proof}
(1) 
First note that the formation of $\mathcal{I}_S(\omega_R(D), \Gamma)$ commutes with localization. 
Therefore, by Remark \ref{basic properties of parameter test} (1), we may assume that $(R,\m)$ is local. 
By the minimality of $\tau_{D}(\omega_R, \Gamma)$, it suffices to show that the submodule $\mathcal{I}_S(\omega_R(D), \Gamma)$ of $\omega_R(D)$ satisfies conditions (a) and (b) in Definition \ref{parameter test defn}. 

To verify the condition (a), by localizing at the generic point of $D$, we may assume that $R$ is an $F$-finite DVR, $S$ is a Dedekind domain and $\Gamma=0$. 
Let $\widehat{R}$ denote the completion of $R$, $\widehat{D}$ denote the flat pullback of $D$ via the canonical morphism $\Spec \widehat{R} \to \Spec R$, and set $\widehat{S}:=S \otimes_R \widehat{R}$. 
The $\widehat{R}$-algebra $\widehat{S}$ is isomorphic to a finite product $S_1 \times \cdots \times S_r$ of complete DVRs $(S_i, \mathfrak{n}_i)$, and $S(-D_S)\widehat{S}$ is a maximal ideal of $\widehat{S}$. After reindexing, we may assume that $S(-D_S)\widehat{S} \cong \mathfrak{n}_1 \times S_2 \times \cdots \times S_r$. 
Then one has 
\[
\mathcal{I}_S(\omega_R(D)) \otimes_R \widehat{R}=\mathrm{Im}\left({\pi_1}_*\omega_{S_1}(D_{S_1}) \to \omega_{\widehat{R}}(\widehat{D})\right)+
\sum_{i=2}^r \mathrm{Im}\left({\pi_i}_*\omega_{S_i} \to \omega_{\widehat{R}}(\widehat{D})\right),
\] 
where $\pi_i:\Spec S_i \to \Spec \widehat{R}$ is the finite surjective morphism induced by $\widehat{R} \to \widehat{S} \to S_i$ and $D_{S_1}$ is the prime divisor on $\Spec S_1$ corresponding to $\mathfrak{n}_1$. 
To verify that $\mathcal{I}_S(\omega_R(D))=\omega_R(D)$, it suffices to show that $\mathrm{Im}({\pi_1}_*\omega_{S_1}(D_{S_1})\to \omega_{\widehat{R}}(\widehat{D}))=\omega_{\widehat{R}}(\widehat{D})$. 
Therefore, we can restrict our attention to the case where $R$ and $S$ are both complete DVRs. 
It follows from Proposition \ref{completion inclusion} and Lemma \ref{test ideal interpretation} that $\tau_D(\omega_R) \subseteq \mathcal{I}_S(\omega_R(D))$. 
Conversely, since $R$ is an $F$-finite DVR and $D$ is the divisor corresponding to the maximal ideal $\m$, it is straightforward to check that $\tau_D(\omega_R)=\omega_R(D)$. 
Consequently, we conclude that $\mathcal{I}_S(\omega_R(D))=\omega_R(D)$. 

It remains to verify that $\mathcal{I}_S(\omega_{R}(D), \Gamma)$ satisfies the condition (b) in Definition \ref{parameter test defn}. 
For any nonzero element $F^e_*c \in F^e_*R$, we have the following commutative diagram:
\[
\xymatrix{
F^e_*\pi_*\omega_S(D_S-\lfloor \pi^*\Gamma \rfloor) \ar@{^{(}->}[r] \ar[d]_{F^e_*\mathrm{Tr}_{\pi}} & F^e_*\pi_*\omega_S(p^e D_S-\lfloor \pi^*\Gamma \rfloor) \ar[rr]^{\quad \pi_*\mathrm{Tr}_{F^e}(F^e_*c \cdot \underline{\hspace{0.5em}})}  \ar[d]_{F^e_*\mathrm{Tr}_{\pi}} &  & \pi_* \omega_S(D_S-\lfloor \pi^*\Gamma \rfloor) \ar[d]_{\mathrm{Tr}_{\pi}} \\ 
F^e_*\omega_R(D) \ar@{^{(}->}[r] & F^e_*\omega_R(p^eD) \ar[rr]^{\mathrm{Tr}_{F^e}(F^e_*c \cdot \underline{\hspace{0.5em}})} &  & \omega_R(D).
}
\]
Since 
\[\Hom_{R}(F^e_*\omega_{R}(p^eD+\lceil (p^e-1) \Gamma \rceil), \omega_{R}(D)) \subseteq \Hom_R(F^e_*\omega_R(p^eD), \omega_R(D))\]
and $\Hom_R(F^e_*\omega_R(p^eD), \omega_R(D))$ is generated by $\mathrm{Tr}_{F^e}:F^e_*\omega_R(p^eD) \to \omega_R(D)$ as an $F^e_*R$-module, 
the commutativity of the above diagram ensures that $\mathcal{I}_S(\omega_{R}(D), \Gamma)$ satisfies the condition (b). 

(2) By \cite[Lemma 4.15]{BST}, there exists a finite separable extension $R'$ of $R$ contained in $R^+$ such that $R'$ is normal and $\nu^*\Gamma$ is Cartier, where $\nu:\Spec R' \to \Spec R$ is the finite surjective morphism induced by the inclusion $R \hookrightarrow R'$. 
Since $\Gamma$ has no component equal to $D$, the morphism $\nu$ is \'etale over the generic point of $D$ by its construction (see the first paragraph of the proof of \cite[Theorem 6.6]{MSTWW}). 
Let $D'$ be the prime divisor on $\Spec R'$ such that $I_D^+ \cap R'=R(-D')$. 
It then follows from \cite[Proposition 6.5]{MSTWW} and Lemma \ref{test ideal interpretation} that 
\begin{align*}
\tau_D(\omega_R, \Gamma)&=\tau_D(R, D+\Gamma-(K_X+D))\\
&=\mathrm{Tr}_{\nu}(\nu_*\tau_{D'}(R', D'+\nu^*\Gamma-(\nu^*K_X+D')-\mathrm{Ram}_{\nu}))\\
&=\mathrm{Tr}_{\nu}(\nu_*\tau_{D'}(R', D'+\nu^*\Gamma-(K_{X'}+D')))\\
&=\mathrm{Tr}_{\nu}(\nu_*\tau_{D'}(\omega_{R'}, \nu^*\Gamma)), 
\end{align*}
where $X=\Spec R$, $X'=\Spec R'$ and $\mathrm{Ram}_{\nu}$ is the ramification divisor of $\nu:X' \to X$. 
On the other hand, for every finite surjective morphism $\rho:\Spec S \to \Spec R'$ with $S$ normal, one has 
\[
\mathrm{Tr}_{\nu \circ \rho}((\nu \circ \rho)_*\omega_S(D_S-\lfloor (\nu \circ \rho)^*\Gamma \rfloor))
=\mathrm{Tr}_{\nu}(\nu_*\mathrm{Tr}_{\rho}(\rho_*\omega_S(D_S-\lfloor \rho^*\nu^*\Gamma \rfloor))). 
\]
Therefore, replacing $R$ with $R'$ and $\Gamma$ with $\Gamma'$, we may assume that $\Gamma$ is a Cartier divisor. 
Furthermore, by Remark \ref{basic properties of parameter test} (3) and the projection formula, we can reduce the problem to the case where $\Gamma=0$.  

Finally, we will prove that there exists a module-finite extension $S$ of $R$ contained in $R^+$ such that $S$ is normal and  $\tau_D(\omega_R)=\mathcal{I}_S(\omega_R(D))$. 
It follows from repeated applications of Lemma \ref{equational lemma}. 
\end{proof}

\begin{lem}\label{equational lemma}
With notation as in Proposition \ref{parameter test=image of trace}, 
let $S$ be a module-finite extension of $R$ contained in $R^+$ with $S$ normal.  
Note that $\tau_D(\omega_R) \subseteq \mathcal{I}_S(\omega_R(D))$ by Proposition \ref{parameter test=image of trace} (1). 
If $\tau_D(\omega_R) \ne  \mathcal{I}_S(\omega_R(D))$, then there exists a module-finite extension $T$ of $S$ contained in $R^+$ such that $T$ is normal and 
\[
\Supp \mathcal{I}_T(\omega_R(D))/\tau_D(\omega_R) \subsetneq \Supp \mathcal{I}_S(\omega_R(D))/\tau_D(\omega_R). 
\]
\end{lem}
\begin{proof}
Let $\eta$ be a minimal prime of $\Supp \mathcal{I}_S(\omega_R(D))/\tau_D(\omega_R)$, 
and $R_{\eta}$, $S_{\eta}$ and $D_{\eta}$ denote the localization of $R$, $S$ and $D$ at $\eta$, respectively.  
Note that taking absolute integral closure commutes with localization, that is, $(R^+)_{\eta} \cong (R_{\eta})^+$ and $(I_D^+)_{\eta} \cong I_{D_{\eta}}^+$. 
Since the formation of $\mathcal{I}_T(\omega_R(D))$ and $\tau_D(\omega_R)$ commutes with localization (see Remark \ref{basic properties of parameter test} (1)), we have the following sequence: 
\[
\omega_{S_{\eta}}(D_{S_{\eta}}) \twoheadrightarrow \mathcal{I}_{S_{\eta}}(\omega_{R_{\eta}}(D_{\eta}))/\tau_{D_{\eta}}(\omega_{R_{\eta}}) \hookrightarrow \omega_{R_{\eta}}(D_{\eta})/\tau_{D_{\eta}}(\omega_{R_{\eta}}). 
\]
By Lemma \ref{test submodule dual}, applying the Matlis dual functor $(-)^{\vee}:=\Hom_{R_{\eta}}(-, E_{R_{\eta}}(R_{\eta}/\eta R_{\eta}))$ yields the sequence 
\[
H^d_{\eta R_{\eta}}(S_{\eta}(-D_{S_\eta})) \hookleftarrow (\mathcal{I}_{S_{\eta}}(\omega_{R_{\eta}}(D_{\eta}))/\tau_{D_{\eta}}(\omega_{R_{\eta}}))^{\vee} \twoheadleftarrow 
Z(D_{\eta}),
\]
where $d=\dim R_{\eta}$. 
Here, to obtain the isomorphism $(\omega_{R_{\eta}}(D_{\eta})/\tau_{D_{\eta}}(\omega_{R_{\eta}}))^{\vee} \cong Z(D_{\eta})$,
we utilized the fact that the formation of $\tau_{D_{\eta}}(\omega_{R_{\eta}})$ commutes with completion (see Remark \ref{basic properties of parameter test} (2)). 
We will show below that there exists a module-finite extension $T$ of $S$ contained in $R^+$ such that $T$ is normal and the image of $(\mathcal{I}_{S_{\eta}}(\omega_{R_{\eta}}(D_{\eta}))/\tau_{D_{\eta}}(\omega_{R_{\eta}}))^{\vee}$ vanishes in $H^d_{\eta R_{\eta}}(T_{\eta}(-D_{T_\eta}))$. 
 By the commutativity of the diagram
 \[
 \xymatrix{
 H^d_{\eta R_{\eta}}(S_{\eta}(-D_{S_\eta})) \ar[d] & (\mathcal{I}_{S_{\eta}}(\omega_{R_{\eta}}(D_{\eta}))/\tau_{D_{\eta}}(\omega_{R_{\eta}}))^{\vee} \ar@{_{(}->}[l] \ar@{>>}[d] \\
 H^d_{\eta R_{\eta}}(T_{\eta}(-D_{T_\eta})) & (\mathcal{I}_{T_{\eta}}(\omega_{R_{\eta}}(D_{\eta}))/\tau_{D_{\eta}}(\omega_{R_{\eta}}))^{\vee}, \ar@{_{(}->}[l]\\
 }
 \]
this vanishing ensures that $(\mathcal{I}_{T_{\eta}}(\omega_{R_{\eta}}(D_{\eta}))/\tau_{D_{\eta}}(\omega_{R_{\eta}}))^{\vee}=0$. 
Consequently, $\eta$ does not lie in the support of $\mathcal{I}_T(\omega_R(D))/\tau_D(\omega_R)$, which implies the assertion of Lemma \ref{equational lemma}. 

Set $N_S:=(\mathcal{I}_{S_{\eta}}(\omega_{R_{\eta}}(D_{\eta}))/\tau_{D_{\eta}}(\omega_{R_{\eta}}))^{\vee}$. 
The $R_{\eta}$-module $N_S$ has finite length by the choice of $\eta$,  and we have the following commutative diagram:
\[
\xymatrix{
Z(D_{\eta}) \ar@{^{(}->}[r] \ar@{>>}[d] &  H^d_{\eta R_{\eta}}(I_{D_{\eta}}) \ar@{>>}[r] \ar[d] & H^d_{\eta R_{\eta}}(R_{\eta}) \ar[d] \\
N_S \ar@{^{(}->}[r]  & H^d_{\eta R_{\eta}}(S_{\eta}(-D_{S_\eta})) \ar@{>>}[r] & H^d_{\eta R_{\eta}}(S_{\eta}). 
}
\]
Since the image of $Z(D_{\eta})$ in $H^d_{\eta R_{\eta}}(R_{\eta})$ is stable under Frobenius action, the image of $N_S$ in $H^d_{\eta R_{\eta}}(S_{\eta})$ is also stable under Frobenius action. 
It then follows from the equational lemma \cite[Lemma 2.2]{HL} (see also \cite{BST}) that the image of $N_S$ vanishes in $H^{d}_{\eta R_{\eta}}(R^+_{\eta})$. 
Noting that $R^+_{\eta}$ is a big Cohen-Macaulay $R_{\eta}$-algebra, and therefore $H^{d-1}_{\eta R_{\eta}}(R^+_{\eta})=0$, we can consider the following commutative diagram with exact rows:
\[
\xymatrix{
& H^{d-1}_{\eta R_{\eta}}(S_{\eta}/S_{\eta}(-{D_{S_\eta}})) \ar[r] \ar[d] & H^d_{\eta R_{\eta}}(S_{\eta}(-D_{S_{\eta}})) \ar[r] \ar[d]^{\alpha} & H^d_{\eta R_{\eta}}(S_{\eta}) \ar[d]^{\beta} \ar[r] & 0\\
0 \ar[r] & H^{d-1}_{\eta R_{\eta}}((R_{\eta}/I_{D_{\eta}})^+) \ar[r]^{f} & H^d_{\eta R_{\eta}}(I_{D_{\eta}}^+)\ar[r]^g & H^d_{\eta R_{\eta}}(R^+_{\eta}) \ar[r] & 0.
}
\]
Simple diagram chasing shows the existence of a finitely generated $R_{\eta}$-submodule $M_+$ of $H^{d-1}_{\eta R_{\eta}}((R_{\eta}/I_{D_{\eta}})^+)$ such that $f(M_+)=\alpha(N_S)$. 
Since $Z(D_{\eta})$ is stable under the map $F_{D_{\eta},0}: H^d_{\eta R_{\eta}}(I_{D_{\eta}}) \to H^d_{\eta R_{\eta}}(I_{D_{\eta}})$ (as defined in Definition \ref{test submodule local}), its image $\alpha(N_S)$ is also stable under the induced map $H^d_{\eta R_{\eta}}(I^+_{D_{\eta}}) \to H^d_{\eta R_{\eta}}(I^+_{D_{\eta}})$. 
The injectivity of $f$ consequently ensures that $M_+$ is stable under Frobenius action.
Applying the equational lemma again, we deduce that $\alpha(N_S) \cong M_+=0$. 
Thus, by the finite generation of $N_S$, there exists a module-finite extension $T$ of $S$ contained in $R^+$ such that $T$ is normal and the image of $N_S$ vanishes in $H^d_{\eta R_{\eta}}(T_{\eta}(-D_{T_\eta}))$. 
\end{proof}

\begin{cor}\label{test ideal finite extension}
With notation as in Setting \ref{setting test submodules}, suppose that $D$ is a prime divisor and 
$K_X+D+\Gamma$ is $\Q$-Cartier. Fix a choice of $I_D^+$. 
\begin{enumerate}
\item For every module-finite extension $S$ of $R$ contained in $R^+$ with $S$ normal, one has 
\[
\tau_D(R,D+\Gamma) \subseteq \mathrm{Im}(\mathrm{Tr}_{\pi}: \pi_*\sO_Y(K_Y-\lfloor \pi^*(K_X+D+\Gamma) \rfloor +D_S) \to K(X)), 
\]
where $Y:=\Spec S \xrightarrow{\pi} X$ is the finite morphism induced by the inclusion $R \hookrightarrow S$ and $D_S$ is the prime divisor such that $S(-D_S)=I_D^+ \cap S$.  
\item There exists a module-finite extension $S$ of $R$ contained in $R^+$ such that $S$ is a normal domain and the equality holds in the inclusion in (1).  
\end{enumerate}
\end{cor}
\begin{proof}
The assertion follows directly from combining Lemma \ref{test ideal interpretation} with Proposition \ref{parameter test=image of trace}. 
\end{proof}

\section{A generalization of plus closure}
In this section, we introduce a generalization of plus closure to give another description of test ideals along divisors. 
For the theory of classical plus closure, the reader is referred to \cite{HH92} and \cite{Sm}. 

We work with the following setting. 
\begin{setting}\label{Setting of char.p>0}
Let $R$ be a $d$-dimensional $F$-finite normal local domain, 
$\Delta$ be an effective $\Q$-Weil divisor and $D$ be a prime divisor on $X:=\Spec R$ such that no component of $\Delta$ is equal to $D$. 
We fix a choice of $I_D^+$, and let $\Lambda$ denote the set of module-finite extensions $R_{\lambda}$ of $R$, contained in $R^+$, such that each $R_{\lambda}$ is a normal domain. 
When $R_{\lambda}$ belongs to $\Lambda$, we write the morphism corresponding to the inclusion $R \hookrightarrow R_{\lambda}$ by $\pi_\lambda: X_{\lambda}:=\Spec R_{\lambda}\to X$. 
\end{setting}

\begin{defn} \label{D lambda}
With notation as in Setting \ref{Setting of char.p>0}, for each $R_{\lambda} \in \Lambda$,  
let $D_\lambda$ denotes the prime divisor on $\Spec R_{\lambda}$ such that $R_{\lambda}(-D_\lambda)=I_D^+\cap R_{\lambda}$.
\begin{enumerate}
\item 
The $R^+$-module $I_D^+(D+\Delta)$ is defined as
\[
I_D^+(D+\Delta)=\varinjlim_{R_{\lambda}} R_{\lambda}(\lfloor \pi_\lambda^*(D+\Delta)-D_{\lambda}\rfloor).
\]
Note that we have a natural inclusion $R\hookrightarrow I_D^+(D+\Delta).$
\item 
Given an ideal $J$ of $R$, the \textit{$(D+\Delta)$-plus closure} $J^{+_D(D+\Delta)}$ of $J$ along $D$ is defined to be the ideal $J(I^+_D(D+\Delta))\cap R$.
\item 
Given an $R$-module $M$, the \textit{$(D+\Delta)$-plus closure} $0_M^{+_D(D+\Delta)}$ of the zero submodule along $D$ is defined to be the kernel of the natural map $M\to M\otimes_R I_D^+(D+\Delta)$. 
\end{enumerate}
\end{defn}

\begin{rem}
An element $x \in R$ belongs to $J^{+_D(D+\Delta)}$ if and only if $\overline{x}$ does to $0_{R/J}^{+_D(D+\Delta)}$, where $\bar{x}$ is the image of $x$ under the canonical surjection $R \to R/J$.  
\end{rem}

\begin{prop}\label{divisorial test ideal equals BCM adjoint ideal}
With notation as in Setting \ref{Setting of char.p>0}, suppose that $K_X+D+\Delta$ is $\Q$-Cartier, that is,  $r(K_X+D+\Delta)=\Div f$ for some integer $r \ge 1$ and some nonzero element $f \in R$.  
Then
\[
0_{H_\m^d(\omega_R)}^{*_D (D+\Delta)}=\ker\left(H_{\m}^d(\omega_R)\xrightarrow{\cdot f^{\frac{1}{r}}} H_{\m}^d(I_D^+)\right), 
\]
where $H_{\m}^d(\omega_R)\xrightarrow{\cdot f^{\frac{1}{r}}} H_{\m}^d(I_D^+)$ is a map induced by the multiplication by $f^{\frac{1}{r}}$. 
\end{prop}

\begin{proof}
By Corollary \ref{test ideal finite extension}, there exists a module-finite extension $R_\lambda \in \Lambda$ such that $f^{1/r} \in R_{\lambda}$ and 
\[
\tau_{D}(R, D+\Delta) =\mathrm{Im}(\Tr_{\pi_\mu}: {\pi_{\mu}}_*\sO_{X_{\mu}}(K_{X_\mu}-\pi_{\mu}^*(K_X+D+\Delta)+D_{\mu}) \to \sO_X)
\]
holds for all $R_{\mu} \in \Lambda$ containing $R_{\lambda}$. 
Taking its Matlis dual, one has 
	\[
0_{H_{\m}^d(\omega_R)}^{*_D(D+\Delta)}=\ker \left(H_{\m}^d(\omega_R)\xrightarrow{\cdot f^{\frac{1}{r}}} H_{\m}^d(I_{D_\lambda})\right).
	\]
Then taking its direct limit yields the desired equality 
	\[
		0_{H_{\m}^d(\omega_R)}^{*_D(D+\Delta)}=\ker \left(H_{\m}^d(\omega_R)\xrightarrow{\cdot f^{\frac{1}{r}}} H_{\m}^d(I_D^+)\right).
	\]
\end{proof}
\begin{lem}\label{Lemma for BCM adjoint ideals}
	In the setting of Proposition \ref{divisorial test ideal equals BCM adjoint ideal},
	\[
		H_{\m}^d(\omega_R\otimes_R I_D^+(D+\Delta))\xrightarrow{\cdot f^{\frac{1}{r}}}H_\m^d(I_D^+)
	\]
	is an isomorphism.
\end{lem}
\begin{proof}
We consider only $R_{\lambda} \in \Lambda$ such that $f^{\frac{1}{r}} \in R_{\lambda}$, that is, $\pi_{\lambda}^*(K_X+D+\Delta)$ is a Cartier divisor on $\Spec R_{\lambda}$.  
The natural injection
	\[
	\omega_{R}\otimes_R R_{\lambda}(\pi_\lambda^*(D+\Delta)-D_\lambda) \hookrightarrow R_\lambda(\pi_\lambda^*(K_X+D+\Delta)-D_\lambda).
	\]
is an isomorphism on the regular locus of $X$, that is, an isomorphism in codimension one, which yields an isomorphism 
\[
    H_\m^d(\omega_{R}\otimes_R R_{\lambda}(\pi_\lambda^*(D+\Delta)-D_\lambda)) \cong H_\m^d(R_\lambda(\pi_\lambda^*(K_X+D+\Delta)-D_\lambda)).
\]
Here, we used that fact that if $M$ is a finitely generated $R$-module supported on a set of codimension $\ge 2$, then $H_\m^d(M)=H_\m^{d-1}(M)=0$. 
Therefore, we have
\begin{align*}
		H_\m^d(\omega_R\otimes_R I_D^+(D+\Delta))
		& \cong  \varinjlim_{R_{\lambda}}H_\m^d(\omega_R\otimes_R R_\lambda(\pi_\lambda^*(D+\Delta)-D_\lambda)) \\
		& \cong  \varinjlim_{R_{\lambda}} H_\m^d(R_\lambda(\pi_\lambda^*(K_X+D+\Delta)-D_\lambda)) \\
		& \cong  H_\m^d(\varinjlim_{R_{\lambda}}(R_\lambda(\pi_\lambda^*(K_X+D+\Delta)-D_\lambda))) \\
		& \cong  H_\m^d(I_D^+\otimes_{R_{\lambda}}R_{\lambda}(\Div f^{\frac{1}{r}})) \\
		& \cong  H_\m^d(I_D^{+}), 
\end{align*}
	where the last isomorphism is induced by the multiplication by $f^{\frac{1}{r}}$. 
\end{proof}

\begin{prop}\label{finitistic divisorial test ideal}
	With notation as in Proposition \ref{divisorial test ideal equals BCM adjoint ideal}, we have
	\[
		\tau_{D}(R,D+\Delta)=\bigcap_{J} (J:J^{+_D(D+\Delta)}),
	\]
	where $J$ runs through all ideals of $R$.
\end{prop}

\begin{proof}
By Proposition \ref{divisorial test ideal equals BCM adjoint ideal} and Lemma \ref{Lemma for BCM adjoint ideals}, 
\begin{eqnarray*}
		0_{H_{\m}^d(\omega_R)}^{*_D(D+\Delta)} &=& \ker(H_{\m}^d(\omega_R)\to H_{\m}^d(\omega_R\otimes_R I_D^+(D+\Delta))) \\
		&=& \ker(H_\m^d(\omega_R)\to H_{\m}^d(\omega_R)\otimes_R I_D^+(D+\Delta)) \\
		&=& 0_{H_\m^d(\omega_R)}^{+_D(D+\Delta)},  
\end{eqnarray*}
where the second equality follows from a general fact that if $M$ and $N$ are modules over a $n$-dimensional Noetherian local ring $(A, \m_A)$, there exists a natural isomorphism $H^n_{\m_A}(M\otimes_A N)\cong H^n_{\m_A}(M)\otimes_A N$.
Since $R$ is excellent and reduced, the assertion follows from an argument very similar to \cite[Proposition 3.3.1 (4)]{DT} and \cite[Proposition 8.23]{HH90}. 
Note that locally excellent Noetherian reduced rings are approximately Gorenstein (see \cite{Hoc77}), which allows us to apply an argument similar to loc.~cit. 
\end{proof}

\section{A characterization of adjoint ideals via ultraproducts}
In this section, we give a characterization of the adjoint ideal $\adj_D(X, D+\Delta)$ via ultraproducts when $K_X+D+\Delta$ is $\Q$-Cartier. 
We work with the following setting. 

\begin{setting}\label{Setting of equal char.0}
Let $(R,\m)$ be a $d$-dimensional normal local domain essentially of finite type over $\C$, 
$\Delta$ be an effective $\Q$-Weil divisor and $D$ be a prime divisor on $X:=\Spec R$ such that no component of $\Delta$ is equal to $D$. Fix a non-principal ultrafilter $\mathcal{F}$ on $\mathcal{P}$ and a field isomorphism $\ulim_p \overline{\F_p}\cong \C$.
Let $(R_p)_{p \in \mathcal{P}}$, $(D_p)_{p \in \mathcal{P}}$ and $(\Delta_p)_{p \in \mathcal{P}}$ be approximations of $R$, $D$ and $\Delta$, respectively. 
Fix choices of $I_{D_p}^+$, which is equivalent to fixing local ring homomorphisms $R_p^+ \to (R_p/I_{D_p})^+$, for almost all $p$. 
\end{setting}

First we generalize Schoutens' ``canonical" big Cohen-Macaulay algebras $\B(R)$ to the pair setting. 
\begin{defn}\label{Beta}
With notation as in Setting \ref{Setting of equal char.0}, the $R^+$-algebra $\B(R)$ is defined as 
\[
\B(R)=\ulim_p R_p^+.
\]
The $\B(R)$-modules $\B(I_D)$ and $\B(I_D,D+\Delta)$ are defined as
\[
\B(I_D)=\ulim_p (I_{D_p}^+), \quad \B(I_D,D+\Delta)=\ulim_p (I_{D_p}^+(D_p+\Delta_p)), 
\]
respectively. 
\end{defn}

\begin{rem}\label{remark non-standard}
Definition \ref{Beta} is an abuse of notation since $\B(I_D)$ and $\B(I_D,D+\Delta)$ depend on the choices of $(I_{D_p}^+)_{p \in \mathcal{P}}$ and are not uniquely determined by $I_D$ and $D+\Delta$. 
If $\sigma: \B(R) \to \B(R/I_D)$ is the homomorphism induced by the fixed local ring homomorphisms $(R_p)^+\to (R_p/I_{D_p})^+$, then
	\[
	0\to \B(I_D) \to \B(R) \xrightarrow{\sigma} \B(R/I_D) \to 0
	\]
is an exact sequence. 
\end{rem}

We define a closure operation in equal characteristic zero, using $\B(I_D,D+\Delta)$. 
\begin{defn}
Let the notation be as in Setting \ref{Setting of equal char.0}. 
\begin{enumerate}
\item Given an ideal $J \subseteq R$, the ideal $J^{\B_D(D+\Delta)} \subseteq R$ is defined to be 
$J \B(I_D,D+\Delta)\cap R$. 
\item Given an $R$-module $M$, the submodule $0_{M}^{\B_D(D+\Delta)}$ is defined to be the kernel of the natural map $M\to M\otimes_{R}\B(I_D,D+\Delta)$. 
\item The following ideals are equal to each other (cf.~\cite[Proposition 8.23]{HH90} and \cite[Proposition 3.3.1]{DT}), and are collectively denoted by $\tau_{\B, D}(R, D+\Delta)$. 
\begin{enumerate}
\item $\bigcap_M \Ann_R 0_M^{\B_D(D+\Delta)}$, where $M$ runs through all $R$-modules. 
\item $\Ann_R 0_E^{\B_D(D+\Delta)}$, where $E=E_R(R/\m)$ is an injective hull of the residue field $R/\m$. 
\item $\bigcap_{J} (J:J^{\B_D(D+\Delta)})$, where $J$ runs through all ideals of $R$.
\end{enumerate}
 \end{enumerate}
\end{defn}

In order to prove the main theorem in this section, we need the following two lemmas. 

\begin{lem}\label{approximation of adjoint ideals}
With notation as in Setting \ref{Setting of equal char.0}, if $K_X+D+\Delta$ is $\Q$-Cartier, then $(\tau_{D_p}(R_p,D_p+\Delta_p))_{p \in \mathcal{P}}$ is an approximation of the adjoint ideal $\adj_D(X,D+\Delta)$. 
\end{lem}
\begin{proof}
Suppose we are given a model over a finitely generated $\Z$-subalgebra  $A$ of $\C$ (see Section   \ref{faithfully flat section} for the definition of models and related notation). 
When $K_X+D+\Delta$ is $\Q$-Cartier, it follows from essentially the same argument as the proof of \cite[Theorem 5.3]{Tak08} that 
\[\adj_D(X,D+\Delta)_{\mu}=\tau_{D_{\mu}}(X_{\mu}, D_{\mu}+\Delta_{\mu})\]
for general closed points $\mu \in \Spec A$. 
Then we obtain the assertion by an argument similar to \cite[Proposition 5.5]{Yam}. 
\end{proof}

\begin{lem}\label{lemma local cohomology injectivity}
With notation as in Setting \ref{Setting of equal char.0}, the natural map 
\[
\beta_D: H_\m^d(\B(I_D)) \to \ulim_p H_{\m_p}^d(I_{D_p}^+)
\]
is injective.
\end{lem}
\begin{proof}
As mentioned in Remark \ref{remark non-standard}, the exact sequences 
\[
0 \to I_{D_p}^+ \to R_p^+ \to (R_p/I_{D_p})^+ \to 0
\]	
for almost all $p$ induce the exact sequence 
\[
0 \to \B(I_D) \to \B(R) \to \B(R/I_D) \to 0.
\]
Note that $R_p^+$ and $(R_p/I_{D_p})^+$ are big Cohen-Macaulay algebras for almost all $p$ by \cite{HH92} and 
that $\B(R)$ and $\B(R/I_D)$ are big Cohen-Macaulay $R^+$-algebras by \cite{Sch04}.  
Thus, we have the following  commutative diagram with exact rows: 
\[
\xymatrix{
0 \ar[r] & H_\m^{d-1}(\B(R/I_D)) \ar[r] \ar[d]^{\alpha} & H_\m^d(\B(I_D)) \ar[d]^{\beta_D} \ar[r] & H_\m^d(\B(R)) \ar[r] \ar[d]^{\gamma} & 0 \\
0 \ar[r]& \ulim_p H_{\m_p}^{d-1}((R_p/I_{D_p})^+) \ar[r] & \ulim_p H_{\m_p}^d(I_{D_p}^+) \ar[r] & \ulim_pH_{\m_p}^d(R_p^+) \ar[r] & 0.
}
\]
Since $\alpha$ and $\gamma$ are injective by \cite[Lemma 4.7]{Yam}, 
so is the homomorphism $\beta_D$.
\end{proof}

The main result in this section is now stated as follows. 
\begin{thm}\label{BCM ajoint ideal equals adjoint ideal}
With notation as in Setting \ref{Setting of equal char.0}, if $K_X+D+\Delta$ is $\Q$-Cartier, then 
\[
0^{\B_D(D+\Delta)}_{H^d_{\m}(\omega_R)}=\Ann_{H^d_{\m}(\omega_R)} \adj_D(X, D+\Delta). 
\]
Taking the annihilator of both sides in $R$ yields the equality 
\[
\tau_{\B, D}(R,D+\Delta) =\adj_D(X,D+\Delta).
\]
\end{thm}

\begin{proof}
First we will prove that $0^{\B_D(D+\Delta)}_{H^d_{\m}(\omega_R)} \subseteq \Ann_{H^d_{\m}(\omega_R)} \adj_D(X, D+\Delta)$. 
It suffices to show that 
$\tau_{\B, D}(R,D+\Delta) \supseteq \adj_D(X,D+\Delta)$, that is, 
$J:J^{\B_D(D+\Delta)} \supseteq \adj_D(R,D+\Delta)$ for every ideal $J \subseteq R$. 
Fix $x\in J^{\B_D(D+\Delta)}$ and $a\in \adj_D(R,D+\Delta)$, and let $(x_p)_{p \in \mathcal{P}}$ and $(a_p)_{p \in \mathcal{P}}$ be approximations of $x$ and $a$, respectively.  
By the definition of $J^{\B_D(D+\Delta)}$, $x_p$ is contained in $J_pI_{D_p}^+(D_p+\Delta_p)$, that is, $x_p \in J_p^{+_{D_p}(D_p+\Delta_p)}$ for almost all $p$. 
On the other hand, by Lemma \ref{approximation of adjoint ideals},  $a_p$ is contained in $\tau_{D_p}(R_p,D_p+\Delta_p)$ for almost all $p$. 
It follows from Proposition \ref{finitistic divisorial test ideal} that $a_p x_p \in J_p$ for almost all $p$, which implies that $a x \in J$. 
Thus, we have the desired containment. 

Next we will prove the reverse containment. 
We may assume that $d \ge 1$. 
Take a log resolution $\mu:Y\to X$ of the pair $(X, D+\Delta)$ and let $Z=\mu^{-1}(\m)$ denote the closed fiber of $\mu$. 
Set 
\[\mathcal{L}=\sO_Y(\mu^*(K_X+D+\Delta)-\mu_*^{-1}D)\] 
and let $\delta:H_\m^d(\omega_R)\to H_Z^d(Y,\mathcal{L})$
be the map induced by the edge maps of the spectral sequence 
$H^p_{\m}(R^q\mu_*\mathcal{L}) \Longrightarrow H^{p+q}_Z(\mathcal{L})$. 
Let $(Y_p)_{p \in \mathcal{P}}$, $(Z_p)_{p \in \mathcal{P}}$ and $(\sL_p)_{p \in \mathcal{P}}$ be approximations of $Y$, $Z$ and $\sL$, respectively. 
Note that one has log resolutions $\mu_p:Y_p\to \Spec R_p$ and $Z_p=\mu_p^{-1}(\m_p)$ for almost all $p$. 
Then we have a commutative diagram
	\[
	\xymatrix@C=40pt{
		& H^{d-1}(X \setminus \{\m\}, \omega_R) \ar[d]^{\gamma} \ar@{->>}[r] & H_\m^d(\omega_R) \ar[d]^{\delta} & \\
		H^{d-1}(Y,\sL)\ar[d] \ar[r] & H^{d-1}(Y\setminus Z,\sL) \ar [r] \ar[d]^-{u^{d-1}} & H_Z^d(Y, \sL) \\
		\ulim_p H^{d-1}(Y_p,\sL_p) \ar[r]^-{\ulim_p\rho_p^{d-1}} & \ulim_p H^{d-1}(Y_p\setminus Z_p,\sL_p),
	}
	\]
where the top horizontal map is surjective and the middle row is exact (see \cite{Sch05} and \cite[Definition 3.67]{Yam} for the definition of $u^{d-1}$).
Similarly, we have the following commutative diagram where the top horizontal map is surjective and the bottom row is exact for almost all $p$:
	\[
	\xymatrix@C=40pt{
		& H^{d-1}(X_p \setminus \{\m_p\}, \omega_{R_p})  \ar[d]^{\gamma_p} \ar@{->>}[r] & H_{\m_p}^d(\omega_{R_p}) \ar[d]^{\delta_p} & \\
		H^{d-1}(Y_p,\sL_p) \ar[r]^-{\rho_p^{d-1}} & H^{d-1}(Y_p\setminus Z_p, \sL_p) \ar[r] & H^d_{Z_p}(Y_p,\sL_p).
	}
	\]

It is enough to show that $\ker \delta \subseteq 0^{\B_D(D+\Delta)}_{H^d_{\m}(\omega_R)}$, because $\delta$ is the Matlis dual of the inclusion $\adj_D(X, D+\Delta) \hookrightarrow R$ and $\ker \delta=\Ann_{H^d_{\m}(\omega_R)} \adj_D(X, D+\Delta)$. 
Suppose $\eta \in \ker \delta$ and take an element $\zeta \in H^{d-1}(X\setminus \{\m\}, \omega_R)$ that maps to $\eta$. 
Let $(\eta_p)_{p \in \mathcal{P}}$ and $(\zeta_p)_{p \in \mathcal{P}}$ be approximations of $\eta$ and $\zeta$, respectively.  
By the commutativity of the first diagram, 
$u^{d-1}(\gamma(\zeta))\in \operatorname{Im}(\ulim_p\rho_p^{d-1})$, 
which implies that $\gamma_p(\zeta_p) \in \operatorname{Im}\rho_p^{d-1}$ for almost all $p$. 
Then, by the commutativity of the second diagram, $\eta_p\in \ker \delta_p$ for almost all $p$. 
Here, note that 
\begin{align*}
\ker \delta_p&=\Ann_{H^d_{\m_p}(\omega_{R_p})} \adj_{D_p}(X_p, D_p+\Delta_p)\\
&=\Ann_{H^d_{\m_p}(\omega_{R_p})} \adj_D(X,D+\Delta)_p\\
&=\Ann_{H^d_{\m_p}(\omega_{R_p})} \tau_{D_p}(R_p, D_p+\Delta_p)\\
&=0_{H_{\m_p}^d(\omega_{R_p})}^{*_{D_p}(D_p+\Delta_p)}\\
&=\ker\left(H_{\m_p}^d(\omega_{R_p})\to H_{\m_p}^d(\omega_{R_p}\otimes_{R_p} I_{D_p}^+(D_p+\Delta_p))\right)
\end{align*}
for almost all $p$, where the third equality follows from the dual form of Lemma \ref{approximation of adjoint ideals}, the fourth from Remark \ref{completion remark} and the fifth from Proposition \ref{divisorial test ideal equals BCM adjoint ideal} and Lemma \ref{Lemma for BCM adjoint ideals}. 
Therefore, the image of $\eta$ vanishes in $\ulim_p H_{\m_p}^d(\omega_{R_p}\otimes_{R_p} I_{D_p}^+(D_p+\Delta_p))$. 
Since $K_X+D+\Delta$ is $\Q$-Cartier, $r(K_X+D+\Delta)=\Div f$ for some integer $r \ge 1$ and some nonzero element $f \in R$. 
Fix a module-finite extension $S$ of $R$ contained in $R^+$ such that $S$ is normal and $f^{1/r} \in S$, and let $\pi:\Spec S \to \Spec R=X$ denote the morphism corresponding to the inclusion $R \hookrightarrow S$. 
We now consider a commutative diagram
{\small
\[
\hspace*{-0.65em}
\xymatrix@C-1.25pc{
H_\m^d(\omega_R) \ar[rrr] \ar[d]_{\cong}  & & & \ulim_p H_{\m_p}^d(\omega_{R_p}) \ar[dd]^{\cdot f^{1/r}}\\
H_\m^d(\omega_R \otimes_R R(D) \otimes_R R(-D)) \ar[d] & & &  \\
H_\m^d(\omega_R \otimes_R S(\pi^*(D+\Delta)) \otimes_S \B(I_D)) \ar[rr]^{\hspace*{5.5em} \cdot f^{1/r}}_{\hspace*{5.5em} \cong} \ar[d] & \quad & H^d_\m(\B(I_D)) \ar[r]^{\hspace*{-1em}\beta_D}  & \ulim_p H^d_{\m_p}(I_{D_p}^+) \\
H_{\m}^d(\omega_R\otimes_R \B(I_D,D+\Delta)) \ar[d]_{\cong} & & &  \\
H^d_{\m}(\ulim_p \omega_{R_p} \otimes_{R_p}  I_{D_p}^+(D_p+\Delta_p))  \ar[rrr] &  & &
\ulim_p H_{\m_p}^d(\omega_{R_p}\otimes_{R_p} I_{D_p}^+(D_p+\Delta_p)) \ar[uu]^{\cong}_{\cdot f^{1/r}},
}
\]
}
where $\beta_D$ is injective by Lemma \ref{lemma local cohomology injectivity} and the isomorphisms in the lower left and lower right are consequences of Lemma \ref{lemma tensor ultraproduct} and Lemma \ref{Lemma for BCM adjoint ideals}, respectively. 
By the commutativity of this diagram, the image of $\eta$ has to be zero in $H_\m^d(\omega_R \otimes_R S(\pi^*(D+\Delta)) \otimes_S \B(I_D))$. Thus, 
\begin{align*}
\eta &\in \ker(H_{\m}^d(\omega_{R})\to H_{\m}^d(\omega_{R}\otimes_{R} \B(I_D, D+\Delta)))\\
&=\ker(H_{\m}^d(\omega_{R}) \to H_{\m}^d(\omega_{R}) \otimes_R  \B(I_D, D+\Delta))\\
&=0^{\B_D(D+\Delta)}_{H_{\m}^d(\omega_{R})}. 
\end{align*}
\end{proof}

\section{Pullback of divisors}
In this section, we discuss how to pullback Weil divisors. 
Our main reference is \cite[Section 2]{dFH}. 
Although morphisms are assumed to be birational in \textit{loc.~cit.}, essentially the same arguments work in our setting. 

\begin{defn}[\textup{cf.~\cite[Section 2]{dFH}}]\label{pullback}
Let $R\hookrightarrow S$ be an injective homomorphism between Noetherian normal domains and $\phi:\Spec S \to \Spec R$ denote the corresponding morphism. 
\begin{enumerate}
\item 
Suppose that $D$ is a Weil divisor on $\Spec R$. 
The \textit{cycle-theoretic pullback} $\phi^{\natural}D$ of $D$ under $\phi$ is the Weil divisor 
\[
\phi^{\natural} D=\sum_{E}v_E(R(-D))E,
\]
where $E$ runs through all prime divisors on $\Spec S$ and $v_E$ is the discrete valuation associated to $E$. 
\item Suppose that $\Gamma$ is an $\R$-Weil divisor on $\Spec R$. 
The \textit{pullback} $\phi^*\Gamma$ of $\Gamma$ under $\phi$ is the $\R$-Weil divisor 
\[
\phi^*\Gamma=\sum_{E}\left(\inf_{m}\frac{v_E(R(\lceil -m\Gamma\rceil))}{m}\right)E,
\]
where $E$ runs through all prime divisors on $\Spec S$ and the infimum is taken over all integers $m \ge 1$.
If $\Gamma$ is $\Q$-Cartier, then this definition coincides with the classical definition of pullback. 
\end{enumerate}
\end{defn}

\begin{rem}\label{pullback rem}
\begin{enumerate}
\item 
$S(-\varphi^{\natural}D)=(R(-D)S)^{**}$, where $(-)^{**}$ denotes the reflexive hull as an $S$-module. 
\item 
If $D_1$ and $D_2$ are Weil divisors on $\Spec R$, then $\phi^{\natural}(D_1+D_2) \le \phi^{\natural}D_1+\phi^{\natural} D_2$ holds and the inequality is strict in general. 
\item 
If $D$ is a Weil divisor, one generally has the inequality $\varphi^*D \le \varphi^{\natural} D$. 
If $\varphi$ is flat, then $\varphi^{\natural}D=\varphi^*D$, which also coincides with the flat pullback of $D$ under $\varphi$. Similarly, if $\varphi$ is finite, then $\varphi^\natural D=\varphi^* D$, which coincides with the usual finite pullback of $D$ under $\varphi$.
\item 
Definition \ref{pullback} can be generalized to the case of dominant morphisms $\varphi:Y \to X$ between normal (not necessarily affine) varieties. 
If $\varphi$ is a  small birational morphism, then $\varphi^*D$ is nothing but the strict transform of $D$ on $Y$ (see \cite[Remark 2.12]{CEMS}). 
\end{enumerate}
\end{rem}

From now on, we work with the following setting. 
\begin{setting}\label{Setting cycle-theoretic pullback}
Let $k$ be an algebraically closed field. 
Suppose that $R \hookrightarrow S$ is an injective $k$-algebra homomorphism between normal domains essentially of finite type over $k$ and $\phi:\Spec S \to \Spec R$ is the corresponding morphism.  
Let $\Lambda$ (resp. $M$) be the set of module-finite extensions $R_{\lambda}$ (resp.~$S_{\mu}$) of  $R$ (resp.~$S$), contained in $R^+$ (resp.~$S^+$), such that each $R_{\lambda}$ (resp.~$S_{\mu}$) is a normal domain. 
When $R_{\lambda}$ (resp.~$S_{\mu}$) belongs to $\Lambda$ (resp.~$M$), 
we write the morphism corresponding to the inclusion $R \hookrightarrow R_\lambda$ (resp.~$S \hookrightarrow S_{\mu}$) by $\pi_\lambda:\Spec R_{\lambda}\to \Spec R$ (resp.~$\rho_\mu:\Spec S_{\mu}\to \Spec S$). 
\end{setting}

\begin{prop}\label{cycle-theoretic pullback prime case}
With notation as in Setting \ref{Setting cycle-theoretic pullback}, 
take $R_{\lambda} \in \Lambda$ and $S_{\mu} \in M$ such that $R_\lambda$ is contained in $S_\mu$ and let $\phi_{\lambda \mu}:\Spec S_{\mu} \to \Spec R_{\lambda}$ denote the corresponding morphism. 
For a Weil divisor $D$ on $\Spec R$, note that $\pi_{\lambda}^* D$ is a Weil divisor on $\Spec R_{\lambda}$ since $\pi_{\lambda}$ is a finite morphism. 
Then one has an inequality 
\[
\rho_{\mu}^*(\varphi^{\natural} D) \ge \varphi_{\lambda \mu}^{\natural} (\pi_{\lambda}^* D)
\]
of Weil divisors on $\Spec S_{\mu}$. 
\end{prop}

\begin{proof}
Let $F_{\mu}$ be a prime divisor on $\Spec S_{\mu}$ and $F=\rho_{\mu}(F_{\mu})$ denote the image of $F_{\mu}$ under $\rho_{\mu}$. Then 
\begin{align*}
\ord_{F_{\mu}}(\varphi_{\lambda \mu}^{\natural} (\pi_{\lambda}^* D))=v_{F_{\mu}}((I_D R_{\lambda})^{**}) &\le v_{F_{\mu}}(I_D)\\
&=v_{F_{\mu}}(I_F)v_F(I_D)\\
&=\ord_{F_{\mu}}(F) \ord_F(\varphi^{\natural} D)\\
&=\ord_{F_{\mu}}(\rho_{\mu}^*(\varphi^{\natural} D)),
\end{align*}
where $(I_D R_{\lambda})^{**}$ is the reflexive hull of $I_D R_{\lambda}$ as an $R_{\lambda}$-module. 
\end{proof}

\begin{rem}
Cycle-theoretic pullback does not commute with finite pullback, that is, the inequality in Proposition \ref{cycle-theoretic pullback prime case} is strict in general. 
For example, let $S=\C[x, y, z]$ be the 3-dimensional polynomial ring over $\C$ and $R=\C[xy^2, xyz, xz^2]$ be a subring of $S$. 
Consider the module-finite extension $R_{\lambda}=\C[\sqrt{x}y, \sqrt{x}z]$ of $R$ and the module-finite extension $S_{\mu}=\C[\sqrt{x}, y, z]$ of $S$. 
\[
\xymatrix{
\Spec  \C[\sqrt{x}, y, z] \ar[r]^{\varphi_{\lambda \mu} \quad} \ar[d]_{\rho_{\mu}}& \Spec \C[\sqrt{x}y, \sqrt{x}z] \ar[d]^{\pi_{\lambda}} \\
\Spec  \C[x, y, z] \ar[r]^{\varphi \quad} & \Spec \C[xy^2, xyz, xz^2] \\
}
\]
Let $D$ be a prime divisor on $\Spec \C[xy^2, xyz, xz^2]$ defined by the prime ideal $(xy^2, xyz)$ of height one. 
Then $\pi_{\lambda}^*D=\Div \sqrt{x}y$ and $\varphi_{\lambda \mu}^{\natural} \pi_{\lambda}^* D=\Div \sqrt{x}+\Div y$. 
On the other hand, $\varphi^{\natural}D=\Div x+\Div y$ and $\rho_{\mu}^*\varphi^{\natural}D=2 \Div \sqrt{x}+\Div y$. 
\end{rem}

\begin{prop}\label{cycltheoretic pullback under pure ring extension}
With notation as in Setting \ref{Setting cycle-theoretic pullback}, 
suppose in addition that $R\hookrightarrow S$ is a pure local homomorphism. 
For a prime divisor $D$ on $\Spec R$, one has  
\[
S(-\phi^{\natural} D)\cap R=R(-D).
\]
\end{prop}
\begin{proof}
First, note that $\varphi$ is surjective since $R \hookrightarrow S$ is pure. 
Pick a prime ideal $\p$ of $S$ lying over $I_D=R(-D)$, and  let $\q$ be a minimal prime ideal of $I_D S$ contained in $\p$. 
Then $\q$ also lies over $I_D$.
By \cite[Theorem 15.1]{Mat}, we have 
\[
    \Ht \q\le \Ht(I_D)+\dim(S_\q/I_DS_\q)=1.
\]
Since $\q$ is a nonzero ideal, it must be a height one prime of $S$ and  $I_{\phi^\natural D} \subseteq \q$.
Therefore,
\[
    I_D \subseteq I_{\phi^\natural D}\cap R \subseteq \q\cap R= I_D,
\]
which completes the proof.
\end{proof}

The following proposition is one of the key ingredients in the study of the behavior of adjoint ideals under pure morphisms. 
\begin{prop}\label{pure ring extension reduction modulo p}
With notation as in Setting \ref{Setting cycle-theoretic pullback}, 
let $D$ be a prime divisor on $\Spec R$, and suppose that the cycle-theoretic pullback $E:=\varphi^{\natural} D$ of $D$ under $\phi$ is a prime divisor and dominates $D$. 
Let $\Gamma$ (resp.~$\Delta$) be an effective $\Q$-Weil divisor on $\Spec R$ (resp.~$\Spec S$)  that has no component equal to $D$ (resp.~$E$), and suppose that $\Delta \ge \phi^*\Gamma$. 
Fix choices of $I_D^+$ and $I_E^+$ such that the following diagram commutes: 
	\[
		\xymatrix{
		0  \ar[r] & I_D^+ \ar[r] \ar[d] & R^+ \ar[r] \ar[d] & (R/I_D)^+ \ar[r] \ar[d] & 0\\
		0 \ar[r] & I_E^+ \ar[r] &  S^+ \ar[r] & (S/I_E)^+ \ar[r] & 0
		}
	\]
Then there exists a natural inclusion
	\[
		I_D^+(D+\Gamma) \hookrightarrow I_E^+(E+\Delta).
	\]
\end{prop}

\begin{proof}
Take $\lambda\in\Lambda$ and $\mu \in M$ such that $R_\lambda$ is contained in $S_\mu$, and let $\varphi_{\lambda \mu}:\Spec S_\mu \to \Spec R_\lambda$ denote the corresponding morphism and  $E_{\mu}$ (resp.~$D_{\lambda}$) denote the prime divisor on $\Spec S_{\mu}$ (resp.~$\Spec R_{\lambda}$) such that $S_{\mu}(-E_{\mu})=I_E^+ \cap S_{\mu}$ (resp.~$R_{\lambda}(-D_{\lambda})=I_D^+ \cap R_{\lambda}$).
It suffices to show the inclusion 
\[
R_{\lambda}(\lfloor \pi_{\lambda}^*(D+\Gamma)-D_\lambda \rfloor)\hookrightarrow S_{\mu}(\lfloor \rho_\mu^*(E+\Delta)-E_\mu \rfloor).
		\]
For any nonzero element $f\in R_\lambda(\lfloor \pi_\lambda^*(D+\Gamma)-D_\lambda \rfloor)$ and any prime divisor $F_{\mu}$ on $\Spec S_\mu$,  
we will show that $\ord_{F_{\mu}}(\Div_{S_\mu} f+\rho_\mu^*(E+\Delta)-E_\mu)\ge 0$. 
First consider the case where $F_{\mu} \neq E_\mu$. 
By assumption, 
\[
\inf_{m} \frac{v_{F_{\mu}}(R(-m\Gamma))}{m}=\ord_{F_{\mu}} (\rho_\mu^*(\phi^*\Gamma)) \le \ord_{F_{\mu}} (\rho_\mu^*\Delta),
\]
where the infimum is taken over all integers $m \ge 1$ such that $m\Gamma$ is an integral Weil divisor. 
Since 
\[
f^m R_{\lambda}(-\pi_{\lambda}^* D)^m R(-m\Gamma)   \subseteq R_\lambda(-mD_\lambda)\subseteq S_\mu
\]
for such $m$, one has 
\begin{align*}
\ord_{F_{\mu}}(\Div_{S_\mu} f+\rho_\mu^*(E+\Delta)-E_\mu)&=\ord_{F_{\mu}}(\Div_{S_\mu} f+\rho_\mu^*E+\rho_\mu^*\Delta)\\
& \ge \ord_{F_{\mu}}(\Div_{S_{\mu}} f+\varphi_{\lambda \mu}^{\natural}\pi^*_{\lambda}D+\rho_{\mu}^*\Delta)\\
& \ge 0,
\end{align*}
where the middle inequality follows from  Proposition \ref{cycle-theoretic pullback prime case}. 

Next we treat the case where $F_{\mu}=E_{\mu}$. 
Since $E$ dominates $D$, the prime divisor $E_{\mu}$ dominates $D_{\lambda}$. 
Also, $\ord_{D_{\lambda}} \pi_{\lambda}^*\Gamma= \ord_{E_{\mu}} \rho_{\mu}^*\Delta=0$ by assumption. 
Therefore, by Proposition \ref{cycle-theoretic pullback prime case}, 
\begin{align*}
\ord_{E_{\mu}}(\Div_{S_\mu}f+\rho_\mu^*(E+\Delta)-E_\mu) &
\ge \ord_{E_{\mu}}(\Div_{S_\mu}f+\varphi_{\lambda \mu}^{\natural}\pi^*_{\lambda}D)-1 \\
& \ge \ord_{D_{\lambda}}(\Div_{R_{\lambda}}f+\pi_{\lambda}^*D)-1 \\
&=\ord_{D_{\lambda}}(\Div_{R_{\lambda}}f+\pi_{\lambda}^*(D+\Gamma)-D_{\lambda})\\
&\ge 0.
\end{align*}
\end{proof}

Cycle-theoretic pullback commutes with taking approximations. 
\begin{prop}\label{reduction modulo p cycle-theoretic pullback}
Suppose that $R\hookrightarrow S$ is an injective local $\C$-algebra homomorphism between normal local rings essentially of finite type over $\C$ and $\varphi:\Spec S \to \Spec R$ is the corresponding morphism. 
Let $D$ be a Weil divisor on $\Spec R$ and $E:=\varphi^{\natural} D$ be the cycle-theoretic pullback of $D$ under $\varphi$. Fix a non-principal ultrafilter $\mathcal{F}$ on $\mathcal{P}$ and a field isomorphism $\ulim_p \overline{\F_p}\cong \C$. If $(\varphi_p:\Spec S_p \to \Spec R_p)_{p \in \mathcal{P}}, (D_p)_{p \in \mathcal{P}}, (E_p)_{p \in \mathcal{P}}$ are approximations of $\varphi, D, E$, respectively, then $E_p$ is the cycle-theoretic pullback of $D_p$ under $\varphi_p$ for almost all $p$.
\end{prop}
\begin{proof}
Suppose that $\p_1,\dots, \p_n$ are all the minimal prime ideals of $I_DS$. 
Let $(\p_{ip})_{p \in \mathcal{P}}$ be an approximation of $\p_i$ for each $i=1, \dots, n$, and then by \cite[Theorem 4.4]{Sch03}, $\p_{1p}, \dots, \p_{np}$ are all the minimal prime ideals of $I_{D_p}S_p$ for almost all $p$. 
By reindexing, if necessary, we may assume that $\Ht \p_i=1$ for $i=1,\dots,m$ and $\Ht \p_i\ge 2$ for $i=m+1,\dots,n$. 
Let $E_i$ denote the prime divisor on $\Spec S$ defined by $\p_i$ and $l_i$ denote the positive integer such that $I_DS_{\p_i}=t_i^{l_i}S_{\p_i}$, where $t_i$ is a uniformizer of the DVR $S_{\p_i}$,  for $i=1,\dots,m$. 
It then follows from Remark \ref{pullback rem} (1) that $E=\sum_{1\le i\le m} l_i E_i$. 

On the other hand, let $E_{ip}$ be the prime divisor on $\Spec S_p$ defined by $\p_{ip}$ for $i=1, \dots, m$ and for almost all $p$. 
Since $I_{D_p}(S_p)_{\p_{ip}}=(I_D S_{\p_i})_p$ and $\Ht \p_i=\Ht \p_{ip}$ for almost all $p$ (see \cite[Theorem 4.5]{Sch03} for the second equality), the cycle-theoretic pullback of $D_p$ under $\varphi_p$ is $\sum_{1\le i \le m} l_i E_{ip}$ for almost all $p$, which completes the proof. 
\end{proof}

The pullback of an approximation of a Weil divisor can be estimated from above by using an approximation of the pullback of this divisor.
\begin{prop}\label{reduction modulo p pullback} 
Suppose that $R\hookrightarrow S$ is an injective local $\C$-algebra homomorphism between normal local domains essentially of finite type over $\C$ and $\phi:\Spec S \to \Spec R$ is the corresponding morphism. 
Let $\Gamma$ be an effective $\Q$-Weil divisor on $\Spec R$, and fix an integer $m \ge 1$ such that $m\Gamma$ is an integral Weil divisor. Fix a non-principal ultrafilter $\mathcal{F}$ on $\mathcal{P}$ and a field isomorphism $\ulim_p \overline{\F_p}\cong \C$. Considering approximations of $\Gamma, \phi^*\Gamma$ and $\phi^\natural m\Gamma$, for every real number $\varepsilon>0$, one has 
\[
(\phi^*\Gamma)_p+\epsilon (\phi^\natural m\Gamma)_p\ge \phi_p^*\Gamma_p
\]
for almost all $p$.
\end{prop}
\begin{proof}
Fix a real number $\varepsilon>0$ and take a sufficiently large integer $n$ so that 
\[
\frac{v_E(R(-mn\Gamma))}{mn}\le \ord_E(\phi^*\Gamma+\epsilon \phi^\natural m\Gamma)
\]
for all prime divisors $E$ on $\Spec S$. Since $v_{E_p}(R_p(-mn \Gamma_p))=v_E(R(-mn\Gamma))$ for almost all $p$, where $(E_p)_{p \in \mathcal{P}}$ is an approximation of $E$, this inequality implies that 
\[
\ord_{E_p}\phi_p^*\Gamma_p \le \frac{v_{E_p}(R_p(-mn \Gamma_p))}{mn}\le \ord_{E_p}((\phi^*\Gamma)_p+\epsilon (\phi^\natural m\Gamma)_p)
\]
for almost all $p$. 
\end{proof}

\section{Faithfully flat descents of adjoint ideals}\label{faithfully flat section}
In this section, we prove the faithfully flat descent property of adjoint ideals. 
First, following an idea from \cite{CEMS}, we extend the correspondence between adjoint ideals and test ideals along divisors to rings with finitely generated anti-canonical algebras. 

\begin{lem}\label{lemma divisorial test ideal = divisorial test submodule}
Suppose that $X$ is an $F$-finite normal integral scheme, $D$ is a reduced divisor and $\Delta$ is an effective $\Q$-Weil divisor on $X$ such that $D$ and $\Delta$ have no common components and the $\sO_X$-algebra $\bigoplus_{i\ge 0}\sO_X(\lfloor -i(K_X+D+\Delta)\rfloor)$ is finitely generated. 
Let $\ba \subseteq \sO_X$ be a coherent ideal sheaf whose zero locus contains no components of $D$ and $t>0$ be a rational number. 
Choose an integer $m \ge 1$ such that $mt$ is an integer, $\m \Delta$ is an integral Weil divisor, and the $m$-th Veronese subring of $\bigoplus_{i\ge 0}\sO_X(\lfloor -i(K_X+D+\Delta)\rfloor)$ is generated in degree one. Then 
	\[
	\tau_D(X,D+\Delta,\ba^t)=\tau_D(\omega_X, (\sO(-m(K_X+D+\Delta))\ba^{tm})^{\frac{1}{m}}). 
	\]
For the definition of the right hand side, see Remark \ref{parameter submodule with ideal}. 
\end{lem}
\begin{proof}
We may assume that $X=\Spec R$, where $R$ is an $F$-finite normal local ring. 
There exists $c\in R^{\circ,D}$ with $\Supp{\Delta}\subseteq \Supp (\Div_X(c))$ such that for every $e_0 \in \N$ and every effective Cartier divisor $B$ with no common components with $D$, one has 
    \[
        \tau_D(X,D+\Delta,\ba^t)=\sum_{e\ge e_0} \mathrm{Tr}_X^e\big(F^e_*\ba^{\lceil tp^e \rceil}\sO_X(\lceil (1-p^e)(K_X+D)-p^e\Delta-\Div_X(c)-B\rceil)\big), 
    \]
where $\mathrm{Tr}_X^e \colon \sO_X((1-p^e)K_X) \to \sO_X$ is the trace map of the $e$-times iterated Frobenius morphism. 
The reader is referred to the proof of \cite[Proposition 6.5]{MSTWW} for the case where $\ba=R$ and $B=0$ (see also \cite[Lemma 3.5]{Tak08}). 
The case where $B \ne 0$ is proved by replacing $c$ with $cd$ for an element $d \in R(-B) \cap R^{\circ, D}$. 
Thanks to this Frobenius trace description of test ideals along divisors, the assertion follows from an argument similar to \cite[Lemma 5.2]{CEMS}. 
\end{proof}

\begin{lem}\label{lemma adjoint ideal for rings with f.g. anti-canonical rings}
Suppose that $X$ is a normal complex variety, $D$ is a reduced divisor and $\Delta$ is an effective $\Q$-Weil divisor on $X$ such that $D$ and $\Delta$ have no common components and 
the $\sO_X$-algebra $\bigoplus_{i\ge 0}\sO_X(\lfloor -i(K_X+D+\Delta)\rfloor)$ is finitely generated.  
Let $\ba\subseteq \sO_X$ be a coherent ideal sheaf whose zero locus contains no components of $D$ and $t>0$ be a real number.  
Let $\rho: X'=\operatorname{\mathbf{Proj}} \bigoplus_{i\ge 0}\sO_X(\lfloor -i(K_X+D+\Delta)\rfloor) \to X$ be the $\Q$-Cartierization of $-(K_X+D+\Delta)$. 
Then
\[
\adj_D(X,D+\Delta,\ba^t)=\rho_* \adj_{\rho^*D}(X',\rho^*(D+\Delta),(\ba \sO_{X'})^t).
\]
\end{lem}
\begin{proof}
Since $\rho$ is a small birational morphism, $\rho^* D$ is reduced.
An argument analogous to that in \cite[Corollary 2.25]{CEMS} is applicable here, by utilizing $\adj^{(m)}_D(X,D+\Delta,\ba^t)$ in the proof of Proposition \ref{single divisor} instead of $\J_m(X,\Delta,\ba^t)$. 
\end{proof}

We now briefly explain the method for reducing triples $(X, D, \ba)$, consisting of varieties, divisors and ideal sheaves, from characteristic zero to positive characteristic. 
Our main reference is \cite[Chapter 2]{HH}.

Let $X$ be a normal variety over a field $k$ of characteristic zero, $D=\sum_i d_i D_i$ be an $\Q$-Weil divisor on $X$ and $\ba \subseteq \sO_X$ be a nonzero coherent ideal sheaf. 
Choosing a suitable finitely generated $\Z$-subalgebra $A$ of $k$, we can construct a scheme $X_A$ of finite type over $A$ and closed subschemes $D_{i, A} \subseteq X_A$ such that 
there exist isomorphisms
\[\xymatrix{
X \ar[r]^{\cong \hspace*{3.5em}} &  X_A \times_{\Spec A} \Spec k\\
D_i \ar[r]^{\cong \hspace*{3.5em}} \ar@{^{(}->}[u] & D_{i, A} \times_{\Spec A} \Spec k, \ar@{^{(}->}[u]\\
}\]
and set $\ba_A:=\rho_* \ba \cap \sO_{X_A}$, where $\rho:X \to X_A$ is the projection. 
We can enlarge $A$ by localizing at a single nonzero element and  subsequently replace $X_A$ and $D_{i,A}$ with their corresponding open subschemes. 
This enables us to assume that $X_A, D_{i, A}$ and $\ba_A$ are flat over $\Spec A$ due to generic freeness. 
Further enlarging $A$ if necessary, we can then assume that $X_A$ is a normal integral scheme, $D_{i,A}$ is a prime divisor on $X_A$ and $\ba_A \sO_X=\ba$. 
We refer to the triple $(X_A, D_A:=\sum_i d_i D_{i, A}, \ba_A)$ as a \textit{model} of $(X, D, \ba)$ over $A$. 
Similarly, we can also consider models of algebras of finite type over $k$. 

Given an closed point $\mu \in \Spec A$, let $X_{\mu}$ (resp.~$D_{i, \mu}$) denote the fiber of $X_A \to \Spec A$ (resp.~$D_{i, A} \to \Spec A$) over $\mu$, and set $D_{\mu}=\sum_i D_{i, \mu}$ and $\ba_{\mu}=\ba_A \sO_{X_{\mu}}$. 
Then $X_{\mu}$ is a scheme of finite type over the finite field $A/\mu$. 
Furthermore, $X_{\mu}$ is a normal variety over $A/\mu$ and $D_{\mu}$ is a $\Q$-Weil divisor on $X_{\mu}$ for general closed points $\mu \in \Spec A$.

\begin{thm}\label{correspondence}
With notation as in Lemma \ref{lemma adjoint ideal for rings with f.g. anti-canonical rings}, suppose that $t$ is a rational number. 
Given a model over a finitely generated $\Z$-subalgebra $A$ of $\C$, we have 
\[
\adj_D(X,D+\Delta,\ba^t)_\mu=\tau_{D_\mu}(X_\mu,D_\mu+\Delta_\mu, \ba_\mu^t)
\]
for general closed points $\mu \in \Spec A$.
\end{thm}
\begin{proof}
By virtue of Lemmas \ref{lemma adjoint ideal for rings with f.g. anti-canonical rings} and \ref{lemma divisorial test ideal = divisorial test submodule}, this follows from an argument similar to \cite[Theorem 6.4]{CEMS}. 
\end{proof}
We remark that flatness is preserved under reduction modulo $p$. 

\begin{prop}\label{proposition redutions mod p>0 flat local}
Let $g:B\to C$ be a $\C$-algebra homomorphism between rings of finite type over $\C$, and $\q$ be a prime ideal of $C$. Moreover, set $\p=\q\cap B$, $R=B_\p$ and $S=C_\q$, and suppose that the induced local ring homomorphism $g_{\p}:R\to S$ is flat. 
Given a model over a finitely generated $\Z$-subalgebra $A$ of $\C$, 
the local ring homomorphism $g_{\p, \mu}:R_\mu \to S_\mu$ is flat for general closed points $\mu \in \Spec A$. 
\end{prop}
\begin{proof}
Since the flat locus of $g:B\to C$ is open, by localizing $C$ at a nonzero element if necessary, we may assume that $g$ is flat. 
Then, enlarging $A$ if necessary,  we may assume that $\operatorname{Tor}_1^{B_A}(B_A/\p_A,C_A)=0$. 
It follows from \cite[Theorem 2.3.5 (e)]{HH} that 
\begin{align*}
	\operatorname{Tor}_1^{R_\mu}(\kappa(\p_\mu),S_\mu) &= \operatorname{Tor}_1^{B_\mu}(B_\mu/\p_\mu,C_\mu) \otimes_{C_\mu} S_\mu \\
		&= \operatorname{Tor}_1^{B_A}(B_A/\p_A,C_A)\otimes_{C_A} S_\mu \\
		&= 0, 
\end{align*}
which implies that $g_{\p, \mu}:R_\mu \to S_\mu$ is flat, 
for general closed points $\mu \in \Spec A$. 
\end{proof}

The following lemma is seemingly well-known to experts. However, we include it here due to the lack of a direct reference.

\begin{lem}\label{Ann flat base change}
	$R\to S$ be a flat local ring homomorphism between Noetherian local rings and $M$ be an $R$-module. If $R$ is complete, then we have
	\[
		(\Ann_{R}M)S=\Ann_{S} (M\otimes_R S).
	\]
\end{lem}
\begin{proof}
For each element $x$ of $M$, tensoring the exact sequence
\[
0 \to \Ann_{R}x \to R\xrightarrow{\cdot x} M
\]
with $S$ yield an exact sequence
\[
0 \to (\Ann_{R}x)S \to S \xrightarrow{\cdot (x\otimes 1)} M\otimes_R S, 
\]
which implies that $(\Ann_{R} x)S=\Ann_{S} (x\otimes 1)$.
Therefore, we have
	\begin{align*}
		(\Ann_R M)S = \left(\bigcap_{x\in M} \Ann_{R} x \right)S 
		&= \bigcap_{x\in M} \left( (\Ann_{R}x) S \right) \\
		&= \bigcap_{x\in M} \Ann_{S}(x\otimes 1) \\
		&= \Ann_{S}(M\otimes_R S),
	\end{align*}
	where the second equality follows from the fact that the homomorphism $R \to S$ is intersection flat by \cite[Proposition 5.7 (e)]{HJ}. 
	Here, a ring homomorphism $A \to B$ is said to be \textit{intersection flat} if it is flat and satisfies the following condition: for every finitely generated $A$-module $M$ and every family of submodules $\{M_{\lambda}\}_{\lambda \in \Lambda}$ of $M$, we have $B \otimes_A \left(\bigcap_{\lambda} M_{\lambda} \right)=\bigcap_{\lambda} (B \otimes_A M_\lambda)$ as submodules of $B \otimes_A M$. 
\end{proof}
\begin{prop} \label{proposition flat descent of divisorial test ideals}
Let $(R,\m)\to (S,\n)$ be a flat local homomorphism between $F$-finite normal local rings of characteristic $p>0$, and let $\phi:\Spec S \to \Spec R$ denote the corresponding morphism. 
Let $D$ be a reduced divisor and $\Gamma$ be an effective $\Q$-Weil divisor on $X:=\Spec R$ such that $D$ and $\Gamma$ have no common components. 
Suppose that the flat pullback $E:=\phi^* D$ of $D$ under $\phi$ is a reduced divisor on $Y:=\Spec S$. 
For any ideal $\ba\subseteq R$ whose zero locus contains no components of $D$ and for any real number $t >0$, one has 
\[
\tau_E(S,E+\phi^* \Gamma,(\ba S)^t)\subseteq \tau_D(R,D+\Gamma, \ba^t)S.
\]
\end{prop}
\begin{proof}
The inclusion $R \hookrightarrow S$ induces the flat injective local ring homomorphism $\widehat{R} \hookrightarrow \widehat{S}$ (see, for example, \cite[0C4G]{Sta}), and let $\widehat{\varphi}: \Spec \widehat{S} \to \Spec \widehat{R}$ denote the corresponding morphism. 
Then we have the commutative diagram
\[
\xymatrix{
\Spec \widehat{S} \ar[r]^{\widehat{\varphi}} \ar[d]_{\iota_S} & \Spec \widehat{R} \ar[d]^{\iota_R} \\
\Spec S \ar[r]^{\varphi} & \Spec R,
}
\]
where $\iota_R : \Spec \widehat{R} \to \Spec R$ and $\iota_S:\Spec \widehat{S} \to \Spec S$ are the canonical morphisms,  and therefore, $\widehat{\varphi}^* \iota_R^*D = \iota_S^*E$. 
Note that both $\iota_R^*D$ and $\iota_S^*E$ are reduced divisors. 
Since the formation of test ideals along divisors commutes with completion, we can reduce the problem to the case where both $R$ and $S$ are complete. 

For each integer $e \ge 1$, the inclusion $R \hookrightarrow S$ induces an inclusion 
\[R((p^e-1)D+\lceil p^e\Gamma \rceil) \hookrightarrow S((p^e-1)E+\lceil p^e f^*\Gamma \rceil).\]
 Additionally, it induces a containment $R^{\circ, D} \subseteq S^{\circ, E}$. 
This follows from the fact that every irreducible component of $E$ dominates an irreducible component of $D$ by the flatness of $\varphi$. 
It is then easy to see that $0_M^{*_D(D+\Gamma, \ba^t)}\otimes_R S \subseteq 0_{M\otimes_R S}^{*_E(E+f^*\Gamma, (\ba S)^t)}$ for all $R$-modules $M$, 
which implies that 
\begin{align*}
\bigcap_{M} \Ann_{S}(0_M^{*_D(D+\Gamma, \ba^t)}\otimes_R S) &\supseteq \bigcap_{M} \Ann_{S}(0_{M\otimes_R S}^{*_E(E+f^*\Gamma, (\ba S)^t)}) \\
& \supseteq \bigcap_{N} \Ann_{S}0_N^{*_E(E+f^*\Gamma, (\ba S)^t)}\\
&=\tau_E(S, E+f^*\Gamma, (\ba S)^t), 
\end{align*}
where $M$ (resp.~$N$) runs through all $R$-modules (resp.~$S$-modules). 
On the other hand, $R \hookrightarrow S$ is intersection flat by the completeness of $R$ and  \cite[Proposition 5.7 (e)]{HJ}, and consequently, 
\begin{align*}
\bigcap_{M} \Ann_{S}(0_M^{*_D(D+\Gamma, \ba^t)}\otimes_R S) &= \bigcap_{M} \left((\Ann_{R}0_{M}^{*_D(D+\Gamma, \ba^t)})S\right) \\
&=\left(\bigcap_{M} \Ann_{R}0_{M}^{*_D(D+\Gamma, \ba^t)}\right)S \\
&=\tau_D(R, D+\Gamma, \ba^t)S,   
\end{align*}
where the first equality follows from Lemma \ref{Ann flat base change}.
Thus, we obtain the desired containment. 
\end{proof}

Here is the main result of this section. 
\begin{thm}\label{faithfully flat thm}
Let $f:Y \to X$ be a faithfully flat morphism between normal complex varieties. 
Let $D$ be a reduced divisor and $\Gamma$ be an effective $\Q$-Weil divisor on $X$ such that $D$ and $\Gamma$ have no common components. 
Suppose that the flat pullback $E:=f^*D$ of $D$ under $f$ is a reduced divisor on $Y$ and the $\sO_X$-algebra $\bigoplus_{i\ge 0} \sO_X(\lfloor -i(K_X+D+\Gamma)\rfloor)$ is finitely generated. 
For any coherent ideal sheaf $\ba\subseteq \sO_X$ whose zero locus contains no components of $D$
and for any real number $t >0$, one has 	
\[
\adj_E(Y,E+f^* \Gamma,(\ba \sO_Y)^t) \subseteq \adj_D(X,D+\Gamma,\ba^t) \sO_Y.
\]
\end{thm}
\begin{proof}
Take an effective $\Q$-Weil divisor $\Delta$ on $Y$ such that $\Delta$ and $E$ have no common components, $K_Y+E+f^*\Gamma+\Delta$ is $\Q$-Cartier and 
\[
\adj_E(Y,E+f^* \Gamma,(\ba \sO_Y)^t)=\adj_E(Y,E+f^* \Gamma+\Delta,(\ba \sO_Y)^t). 
\]
After a small perturbation, we may assume that $t$ is a rational number. 
Given that the question is local, we can further assume that $X=\Spec R$ and $Y=\Spec S$, where $R$ and $S$ are normal local rings essentially of finite type over $\C$. 
Given a model over a finitely generated $\Z$-subalgebra  $A$ of $\C$, it follows from Theorem \ref{correspondence} that 
\begin{align*}
\adj_D(X,D+\Gamma, \ba^t)_{\mu}&=\tau_{D_{\mu}}(X_{\mu}, D_{\mu}+\Gamma_{\mu}, \ba_{\mu}^t), \\
\adj_E(Y,E+f^* \Gamma+\Delta, (\ba S)^t)_{\mu}&=\tau_{E_{\mu}}(Y_{\mu}, E_{\mu}+f_{\mu}^* \Gamma_{\mu}+\Delta_{\mu}, (\ba_{\mu} S_{\mu})^t)
\end{align*}
for general closed points $\mu \in \Spec A$. 
Therefore, it is enough to show that 
\[
\tau_{E_{\mu}}(Y_{\mu}, E_{\mu}+f_{\mu}^* \Gamma_{\mu}+\Delta_{\mu}, (\ba_{\mu} S_{\mu})^t) \subseteq \tau_{D_{\mu}}(X_{\mu}, D_{\mu}+\Gamma_{\mu}, \ba_{\mu}^t) S_{\mu}
\]
for general closed points $\mu \in \Spec A$. 
By observing that all assumptions are preserved after reduction to characteristic $p \gg 0$, such as $f_{\mu}$ being flat for general closed points $\mu \in \Spec A$ by Proposition \ref{proposition redutions mod p>0 flat local}, this is a direct consequence of Proposition \ref{proposition flat descent of divisorial test ideals}. 
\end{proof}

The assertion of Theorem \ref{faithfully flat thm} does not hold for pure morphisms. 
\begin{eg}
Let $S=\C[x,y,z]_{(x,y,z)}$ be a localization of the three-dimensional polynomial ring over the field $\C$ of complex numbers, equipped with an action of the multiplicative group $G=\C^{\times}$ defined by 
	\[
	t:\left\{
	\begin{array}{l}
		x \mapsto t^2 x \\
		y \mapsto t^{-1}y \\
		z \mapsto t^{-1}z,
	\end{array}
	\right.
	\]
where $t \in G$. 
Then the subring $R:=S^G$ of invariants under the action of $G$ is described as 
\[\C[xy^2,xz^2,xyz]_{(xy^2,xz^2,xyz)} \cong (\C[u,v,w]/(uv-w^2))_{(u,v,w)}.\]
Note that the inclusion $R\hookrightarrow S$ is pure and not flat.
Writing $X:=\Spec R$, $Y:=\Spec S$, and $\m$ for the maximal ideal of $R$, we observe that $\J(X,\m)=\m$ and $\J(Y,\m \sO_Y)=(xy,xz)\sO_Y$, which implies that 
\[
	\J(Y,\m \sO_Y)\cap \sO_X \subseteq \J(X,\m), \quad 
	\J(Y,\m \sO_Y) \nsubseteq \J(X,\m)\sO_Y.
\]
\end{eg}

\section{Behavior of adjoint ideals under pure morphisms}
In this section, we study the behavior of adjoint ideals under pure morphisms.
This gives a generalization of \cite[Theorem 1.2]{Yam}. 

In the first half of this section, we work with the following setting. 
\begin{setting}\label{setting behavior of adjoint ideals}
Let $(R,\m)$ be a $d$-dimensional normal local domain essentially of finite type over $\C$, 
$\Delta$ be an effective $\Q$-Weil divisor and $D$ be a prime divisor on $X:=\Spec R$ such that no component of $\Delta$ is equal to $D$. 
Let $\ba$ be an ideal of $R$ not contained in $R(-D)$ and $t >0$ be a real number. Fix a non-principal ultrafilter $\mathcal{F}$ on $\mathcal{P}$ and a field isomorphism $\ulim_p \overline{\F_p}\cong \C$.
\end{setting}

First we generalize the definition of $\tau_{\B, D}(R, D+\Delta)$ to the case of triples. 
\begin{defn}
Let the notation be as in Setting \ref{setting behavior of adjoint ideals}. 
\begin{enumerate}
\item 
Given an $R$-module $M$, the submodule $0_{M}^{\B_D(D+\Delta,\ba^t)}$ of $M$ is defined as 
\[
0_{M}^{\B_D(D+\Delta,\ba^t)}=\bigcap_{n \ge 1}\bigcap_{f\in\ba^{\lceil tn \rceil} \cap R^{\circ,D}}0_{M}^{\B_D(D+\Delta+\frac{1}{n}\Div f)}, 
\]
where the first intersection is taken over all positive integers $n$ and the second intersection is taken over all nonzero elements $f \in \ba^{\lceil tn \rceil}\cap R^{\circ,D}$. 
\item 
The following ideals are equal to each other (cf.~\cite[Proposition 8.23]{HH90}), and are collectively denoted by $\tau_{\B,D}(R,D+\Delta,\ba^t)$. 
\begin{enumerate}
\item  $\bigcap_M \Ann_R 0_{M}^{\B_D(D+\Delta,\ba^t)}$, where $M$ runs through all $R$-modules.
\item $\Ann_R 0_{E}^{\B_D(D+\Delta,\ba^t)}$, where $E=E_R(R/\m)$ is an injective hull of the residue field $R/\m$. 
\end{enumerate}
\end{enumerate}
\end{defn}

\begin{lem}\label{sum of adjoint ideals}
With notation as in Setting \ref{setting behavior of adjoint ideals}, if $K_X+D+\Delta$ is $\Q$-Cartier, then 
\[
\adj_D(X,D+\Delta,\ba^t)=\sum_{n \ge 1}\sum_{f\in \ba^{\lceil tn \rceil}\cap R^{\circ,D}} \adj_D(X,D+\Delta+\frac{1}{n}\Div f), 
\]
where the first summation is taken over all positive integers $n$ and the second summation is taken over all nonzero elements $f \in \ba^{\lceil tn \rceil}\cap R^{\circ,D}$. 
\end{lem}

\begin{proof}
It is clear that the right hand side is contained in the left hand side. We will show the reverse containment. 
First note that the filtration of adjoint ideals $\adj_D(X,D+\Delta,\ba^t)$ is right continuous in $t$, that is, $\adj_D(X,D+\Delta,\ba^t)=\adj_D(X, D+\Delta,\ba^{t+\varepsilon})$ for all $0 \le \varepsilon \ll 1$.  Therefore, we may assume that $t$ is a rational number. 
Let $f_1,\dots,f_l$ be a system of generators for $\ba$ such that $f_i\notin R(-D)$ for each $i=1,\dots,l$. 
Since the adjoint ideal $\adj_D(X,D+\Delta,\ba^t)$ coincides after reduction to characteristic $p \gg 0$ with the test ideal $\tau_D(R, D+\Delta, \ba^t)$ along $D$ by \cite{Tak08}, it follows from an argument similar to the proof of \cite[Theorem 3.2]{Tak06} that 
\[
\adj_D(X,D+\Delta,\ba^t)=\sum_{\lambda_1+\dots+\lambda_l=t}\adj_D(X,D+\Delta+\lambda_1\Div f_1+\dots+\lambda_l \Div f_l),
\]
where the summation is taken over all nonnegative rational numbers $\lambda_1, \dots, \lambda_l$ with $\lambda_1+\dots+\lambda_l=t$. 
Fix such nonnegative rational numbers $\lambda_1,\dots, \lambda_l$ and choose an integer $m \ge 1$ so that $m\lambda_i$ is an integer for each $i=1, \dots, l$. 
Then $f:=f_1^{m\lambda_1}\dots f_l^{m\lambda_l}$ is an element of $\ba^{mt}\cap R^{\circ,D}$ and $\frac{1}{m}\Div f=\lambda_1 \Div f_1+\dots +\lambda_l \Div f_l$. Thus, 
\[
\adj_D(X,D+\Delta+\lambda_1\Div f_1 +\dots+\lambda_l \Div f_l )\subseteq \sum_{n\ge 1}\sum_{f\in \ba^{\lceil tn \rceil}\cap R^{\circ,D}} \adj_D(X,D+\Delta+\frac{1}{n}\Div f),
\]
which completes the proof.
\end{proof}

We can now generalize Theorem \ref{BCM ajoint ideal equals adjoint ideal} to the case of triples. 

\begin{prop}\label{BCM adjoint ideals for triples equal adjoint ideals}
With notation as in Setting \ref{setting behavior of adjoint ideals}, if $K_X+D+\Delta$ is $\Q$-Cartier, then 
\[
\tau_{\B,D}(R,D+\Delta,\ba^t)=\adj_D(X,D+\Delta,\ba^t).
\]
\end{prop}
\begin{proof}
It follows from Theorem \ref{BCM ajoint ideal equals adjoint ideal}, Lemma \ref{sum of adjoint ideals} and Matlis duality that 
\begin{align*}
\tau_{\B,D}(R,D+\Delta,\ba^t) &=\Ann_R 0_{H_\m^d(\omega_R)}^{\B_D(D+\Delta,\ba^t)} \\
&= \Ann_R \left(\bigcap_{n \ge 1}\bigcap_{f \in \ba^{\lceil tn \rceil} \cap R^{\circ, D}}\Ann_{H_\m^d(\omega_R)} \adj_D(X,D+\Delta+\frac{1}{n}\Div f) \right)\\
&= \Ann_R \Ann_{H_\m^d(\omega_R)}\left(\sum_{n \ge 1}\sum_{f\in\ba^{\lceil tn \rceil} \cap R^{\circ, D}} \adj_D(X,D+\Delta+\frac{1}{n}\Div f)\right) \\
&= \Ann_R \Ann_{H_\m^d(\omega_R)}\adj_D(X,D+\Delta,\ba^t) \\
&= \adj_D(X,D+\Delta,\ba^t).
\end{align*}
\end{proof}
\begin{thm}\label{adjoint ideals under pure ring extensions}
With notation as in Setting \ref{setting behavior of adjoint ideals}, let $R\hookrightarrow S$ be a pure local $\C$-algebra homomorphism between normal local domains essentially of finite type over $\C$, and $\varphi: Y:=\Spec S \to \Spec R=X$ denote the corresponding morphism. 
Suppose that $K_X+D+\Delta$ is $\Q$-Cartier and the cycle-theoretic pullback $E:=\varphi^{\natural} D$ of $D$ under $\varphi$ is a prime divisor. 
Then 
\[
\adj_E(Y,E+\phi^*\Delta, \ba S^t)\cap R \subseteq \adj_D(X,D+\Delta,\ba^t).
\]
\end{thm}
\begin{proof}
Choose an integer $m \ge 1$ such that $m \Delta$ is a Weil divisor. 
Note by Proposition \ref{cycltheoretic pullback under pure ring extension} that $\varphi^{\natural} m\Delta$ has no component equal to $E$. 
We take an effective Cartier divisor $G$ on $Y$ whose support contains that of $\phi^{\natural} m \Delta$ and which has no component equal to $E$. 
We also take an effective $\Q$-Weil divisor $\Gamma$ on $Y$ with no components supported on $E$ such that 
$K_Y+E+\varphi^*\Delta+\Gamma$ is $\Q$-Cartier and 
\[\adj_E(Y,E+\phi^*\Delta, \ba S^t)=\adj_E(Y,E+\phi^*\Delta+\Gamma, \ba S^t).\]
Let $(R_p)_{p \in \mathcal{P}}$, $(D_p)_{p \in \mathcal{P}}$ and $(E_p)_{p \in \mathcal{P}}$ be approximations of $R$, $D$ and $E$, respectively. 
Fix choices of $I_{D_p}^+$ and $I_{E_p}^+$
so that the diagram
	\[
	\xymatrix{
		0 \ar[r] & I_{D_p}^+ \ar[r] \ar[d]& R_p^+ \ar[r] \ar[d] & (R/I_D)_p^+ \ar[r] \ar[d] & 0\\
		0 \ar[r] & I_{E_p}^+ \ar[r] & S_p^+ \ar[r] & (S/I_E)_p^+ \ar[r] & 0
	}
	\]
commutes for almost all $p$.
Considering approximations of $\Delta, \Gamma$ and $G$, and applying Proposition \ref{reduction modulo p injectivity}, Proposition \ref{pure ring extension reduction modulo p} and Proposition \ref{reduction modulo p pullback}, we find that the following inclusion holds for any rational number $\epsilon>0$: 
\[
I_{D_p}^+(D_p+\Delta_p) \hookrightarrow I_{E_p}^+(E_p+(\phi^*\Delta)_p+\epsilon G_p + \Gamma_p)
\]
for almost all $p$. 
Furthermore, for every nonzero element $f \in R^{\circ,D}$ and every rational number $s>0$, this inclusion induces an inclusion 
\[
I_{D_p}^+(D_p+\Delta_p+s\Div_{R_p} f_p) \hookrightarrow I_{E_p}^+(E_p+(\phi^*\Delta)_p+\epsilon G_p+\Gamma_p+s\Div_{S_p}f_p)
\]
for almost all $p$, and consequently, an inclusion  
\[
\B(I_D,D+\Delta+s\Div_R f) \hookrightarrow \B(I_E,E+\phi^*\Delta+\epsilon G+\Gamma+s\Div_S f).
\]

Given an $R$-module $M$, we now have the following commutative diagram: 
\[
\xymatrix{
M \ar@{^{(}->}[r] \ar[d] & M\otimes_R S \ar[d]\\
M\otimes_R \B(I_D,D+\Delta+s\Div_R f) \ar[r] & M\otimes_R \B(I_E,E+\phi^*\Delta+\epsilon G+\Gamma+s\Div_S f),
}
\]
where the upper horizontal map is injective due to the purity of the inclusion $R\hookrightarrow S$. 
Therefore, $0_{M}^{\B_D(D+\Delta+s\Div_R f)}$ can be viewed as a submodule of $0_{M\otimes_R S}^{\B_E(E+\phi^*\Delta+\epsilon G+\Gamma+s\Div_Sf)}$. 
When combined with Proposition \ref{BCM adjoint ideals for triples equal adjoint ideals}, this yields 
\begin{align*}
\adj_D(X,D+\Delta,\ba^t) &= \bigcap_{M} \Ann_R \left(\bigcap_{n\ge 1}\bigcap_{f\in \ba^{\lceil tn \rceil}\cap R^{\circ,D}} 0_{M}^{\B_D(D+\Delta+\frac{1}{n} \Div_R f)} \right) \\
& \supseteq  \bigcap_{M} \Ann_R \left(\bigcap_{n\ge 1}\bigcap_{f\in \ba^{\lceil tn \rceil}\cap R^{\circ,D}} 0_{M\otimes_R S}^{\B_E(E+\phi^*\Delta+\epsilon G+\Gamma+\frac{1}{n}\Div_S f)}\right) \\
& \supseteq  \bigcap_{N} \Ann_S \left(\bigcap_{n\ge 1}\bigcap_{f\in \ba^{\lceil tn \rceil}\cap R^{\circ,D}} 0_{N}^{\B_E(E+\phi^*\Delta+\epsilon G+\Gamma+\frac{1}{n}\Div_S f)}\right) \cap R,\\
\end{align*}
where $M$ and $N$ run through all $R$-modules and all $S$-modules, respectively. 
On the other hand, by an argument similar to the proof of \cite[Proposition 3.9]{PR} (cf.~\cite[Proposition 8.23]{HH90}), we have 
\begin{align*}
&\bigcap_{N} \Ann_S \left(\bigcap_{n\ge 1}\bigcap_{f\in \ba^{\lceil tn \rceil}\cap R^{\circ,D}} 0_{N}^{\B_E(E+\phi^*\Delta+\epsilon G+\Gamma+\frac{1}{n}\Div_S f)}\right) \\
=&\Ann_S \left(\bigcap_{n\ge 1}\bigcap_{f\in \ba^{\lceil tn \rceil}\cap R^{\circ,D}} 0_{H_\n^e(\omega_S)}^{\B_E(E+\phi^*\Delta+\epsilon G+\Gamma+\frac{1}{n}\Div_S f)}\right),
\end{align*}
where $e=\dim S$, $\n$ is the maximal ideal of $S$ and $N$ runs through all $S$-modules. 
It follows from Theorem \ref{BCM ajoint ideal equals adjoint ideal} that 
\begin{align*}
&\bigcap_{n \ge 1}\bigcap_{f\in \ba^{\lceil tn \rceil}\cap R^{\circ,D}} 0_{H_\n^e(\omega_S)}^{\B_E(E+\phi^*\Gamma+\epsilon G +\Delta+\frac{1}{n}\Div_S f)} \\
=& \bigcap_{n \ge 1}\bigcap_{f\in \ba^{\lceil tn \rceil}\cap R^{\circ,D}} \Ann_{H_\n^e(\omega_S)}\adj_E(Y,E+\phi^*\Delta+\epsilon G+\Gamma+\frac{1}{n}\Div_S f) \\
=& \Ann_{H_\n^e(\omega_S)}\left(\sum_{n \ge 1}\sum_{f\in \ba^{\lceil tn \rceil}\cap R^{\circ, D}}\adj_E(Y,E+\phi^*\Delta+\epsilon G+\Gamma+\frac{1}{n}\Div_S f)\right) \\
=& \Ann_{H_\n^e(\omega_S)} \adj_E(Y,E+\phi^*\Delta+\epsilon G+\Gamma,(\ba S)^t), 
\end{align*}
with the last equality deduced from essentially the same argument as the proof of Lemma \ref{sum of adjoint ideals} by noting that $R^{\circ, D} \subseteq S^{\circ, E}$. 
Summing up the above inclusions and applying Matlis duality (see, for example, \cite[Lemma 3.3]{Hara}), we obtain 
\begin{align*}
\adj_D(X,D+\Delta,\ba^t) &\supseteq \left(\Ann_S \Ann_{H_\n^e(\omega_S)} \adj_E(Y,E+\phi^*\Delta+\epsilon G+\Gamma,(\ba S)^t)\right) \cap R\\
&= \adj_E(Y,E+\phi^*\Delta+\epsilon G+\Gamma,(\ba S)^t) \cap R. 
\end{align*} 
As $\varepsilon$ approaches zero, the limit results in the desired inclusion
\begin{align*}
\adj_D(X,D+\Delta,\ba^t) & \supseteq \adj_E(Y,E+\phi^*\Delta+\Gamma,(\ba S)^t) \cap R\\
& =\adj_E(Y,E+\phi^*\Delta,(\ba S)^t) \cap R. 
\end{align*}
\end{proof}

We now shift our focus to a global setting. First, we recall the definition of purity in the non-affine context.

\begin{defn}[{\cite[Appendix]{Zhu}}]
A morphism $f:Y\to X$ between Noetherian schemes is said to be {\it pure} if for all $x\in X$, there exists $y\in Y$ such that $f(y)=x$ and the local ring homomorphism $\sO_{X,x}\to \sO_{Y,y}$ is pure.
\end{defn}
\begin{rem}
If $X=\Spec A$ and $Y=\Spec B$ are affine schemes, then $f:Y \to X$ is pure if and only if the induced ring homomorphism $A \to B$ is pure by \cite[Lemma 2.2]{HH95}.
\end{rem}

In the global setting, Theorem \ref{adjoint ideals under pure ring extensions} can be reformulated as follows. 
\begin{cor}\label{adjoint ideal under pure morphism in Q-Cartier case}
Let $f:Y \to X$ be a pure morphism between normal complex varieties, $D$ be a reduced divisor 
and $\Gamma$ be an effective $\Q$-Weil divisor on $X$ that has no common components with $D$. 
Suppose that $K_X+D+\Gamma$ is $\Q$-Cartier and the cycle-theoretic pullback $E:=f^{\natural}D$ of $D$ under $f$ is a disjoint union of prime divisors on $Y$. 
For any coherent ideal sheaf $\ba \subseteq \sO_X$ whose zero locus contains no components of $D$ and for any real number $t>0$, one has 
\[
f_{*}\adj_E(Y,E+f^*\Gamma,(\ba \sO_Y)^t) \cap \sO_X \subseteq \adj_D(X,D+\Gamma,\ba^t).
\]
\end{cor}
\begin{proof}
Since $E$ is a disjoint union of prime divisors,  the same holds true for $D$. 
Considering that the question is local, we may assume that both $X$ and $Y$ are spectra of local rings and that both $D$ and $E$ are prime divisors. 
The assertion is then simply Theorem \ref{adjoint ideals under pure ring extensions}. 
\end{proof}

Next we consider the case where $K_X+D+\Delta$ is not necessarily $\Q$-Cartier but the anti-log-canonical ring $\bigoplus_{i \ge 0}\sO_X(\lfloor -i(K_X+D+\Delta) \rfloor)$ is finitely generated. 
\begin{lem}\label{lemma finite generation}
Let $Y$ be a normal affine variety such that the $\sO_Y$-algebra $\bigoplus_{i\ge 0} \sO(iB)$ is finitely generated  for every reduced divisor $B$ on $Y$ and let $V\subseteq Y$ be an open subset. 
\begin{enumerate}
\item 
Let $V_1=\Spec H^0(\sO_V)$. 
Then the natural morphism $V\to V_1$ is an open immersion whose complement has codimension greater than or equal to 2. 
\item 
Suppose that $E$ is a prime divisor and $F$ is an effective $\R$-Weil divisor on $Y$ such that no component of $F$ is equal to $E$,  with their strict transforms on $V_1$ denoted by $E_1$ and $F_1$, respectively. 
Let $\mathfrak{b}\subseteq \sO_Y$ be an ideal not contained in $\sO_Y(-E)$ and $t>0$ be a real number. 
For any element 
\[c\in \adj_E(Y,E+F,\mathfrak{b}^t),\]
there exists an effective $\R$-Weil divisor $\Delta_1$ on $V_1$ such that  
$K_{V_1}+E_1+F_1+\Delta_1$ is $\Q$-Cartier, no component of $\Delta_1$ is equal to $E_1$, and 
\[c \in \adj_{E_1}(V_1,E_1+F_1+\Delta_1,(\mathfrak{b}\sO_{V_1})^t).\]
\end{enumerate}
\end{lem}

\begin{proof}
The assertion follows from an argument similar to that in \cite[Lemma 2.6]{Zhu}, but we include the proof here for the reader's convenience.

Let $D$ be the divisorial part of $Y \setminus V$, considered as a reduced divisor on $Y$, and then take the $\Q$-Cartierization $Y':=\Proj \bigoplus_{i \ge 0}\sO_X(iD) \xrightarrow{\rho} Y$ of $D$. 
Since the strict transform $D':=\rho^{-1}_*D$ of $D$ is $\rho$-ample, the complement $Y' \setminus D'$ is affine. 
By the choice of $D$, we note that $V \cong \rho^{-1}(V)$ and  $\rho^{-1}(V)$ is an open subset of $Y' \setminus D'$ whose complement has codimension greater than or equal to 2. 
Consequently, we obtain the isomorphism
\[
V_1=\Spec H^0(\sO_V) \cong \Spec H^0(\sO_{Y' \setminus D'})=Y' \setminus D'.
\]
Therefore, the complement of the inclusion $V \cong \rho^{-1}(V) \hookrightarrow Y' \setminus D' \cong V_1$ also has codimension greater than or equal to 2. 

Choose an effective $\R$-Weil divisor $\Delta$ on $Y$ such that $\Delta$ has no component equal to $E$, 
$K_Y+E+F+\Delta$ is $\Q$-Cartier and $\adj_E(Y, E+F, \mathfrak{b}^t)=\adj_E(Y, E+F+\Delta, \mathfrak{b}^t)$. 
Let $E', F'$, and $\Delta'$ denote the strict transforms of $E, F$, and $\Delta$ on $Y'$, respectively, and take a log resolution $\pi:\widetilde{Y} \to Y'$ of $(Y', E'+F'+\Delta', \mathfrak{b}\sO_{Y'})$ such that the strict transform $\pi^{-1}_*E'$ of $E'$ is smooth and $\mathfrak{b}\sO_{\widetilde{Y}}=\sO_{\widetilde{Y}}(-\widetilde{B})$ is invertible. 
Then the condition $c \in \adj_E(Y, E+F, \mathfrak{b}^t)$ is equivalent to the inequality 
\[
K_{\widetilde{Y}}-\lfloor \pi^*(K_{Y'}+E'+F'+\Delta')+t \widetilde{B} \rfloor +\pi^{-1}_*E'+\Div_{\widetilde{Y}} c \ge 0. 
\]
Viewing $V_1$ as a open subset of $Y'$, we set $\widetilde{V}:=\pi^{-1}(V_1)$ and $\pi_{\widetilde{V}}:=\pi|_{\widetilde{V}}: \widetilde{V} \to V_1$. 
Restricting the above inequality to $\widetilde{V}$ yields the inequality 
\[
K_{\widetilde{V}}-\lfloor \pi_{\widetilde{V}}^*(K_{V_1}+E_1+F_1+\Delta_1) +t \widetilde{B}|_{\widetilde{V}} \rfloor+{\pi_{\widetilde{V}}}^{-1}_* E_1+\Div_{\widetilde{V}} c \ge 0, 
\]
where $\Delta_1$ is the strict transform of $\Delta$ on $V_1$. 
This implies that 
\[
c \in \adj_{E_1}(V_1,E_1+F_1+\Delta_1,(\mathfrak{b}\sO_{V_1})^t) \subseteq \adj_{E_1}(V_1,E_1+F_1,(\mathfrak{b}\sO_{V_1})^t).
\]
\end{proof}

\begin{thm}\label{adjoint ideals under pure ring extensions in rings with finitely generated anti-canonical algebra}
Let $f:Y\to X$ be a pure morphism between normal complex varieties, $D$ be a reduced divisor 
and $\Gamma$ be an effective $\Q$-Weil divisor on $X$ that has no common components with $D$. 
Suppose that the $\sO_X$-algebra $\bigoplus_{i \ge 0}\sO_X(\lfloor -i(K_X+D+\Gamma) \rfloor)$ is finitely generated, the $\sO_Y$-algebra $\bigoplus_{i\ge 0}\sO_Y(iB)$ is finitely generated for every reduced divisor $B$ on $Y$, and the cycle-theoretic pullback $E:=f^{\natural}D$ of $D$ is a disjoint union of prime divisors on $Y$.  
For any coherent ideal sheaf $\ba\subseteq \sO_X$ whose zero locus contains no components of $D$ and for any real number $t>0$, one has  
\[
f_*\adj_E(Y,E+f^*\Gamma,(\ba \sO_Y)^t) \cap \sO_X \subseteq \adj_D(X,D+\Gamma,\ba^t).
\]
\end{thm}
\begin{proof}
Since the question is local, we may assume that $X$ and $Y$ are both affine and $D$ and $E$ are both prime divisors.  
We use a similar strategy to the proof of \cite[Lemma 2.8]{Zhu}. 
Take any nonzero element
\[
c\in \adj_E(Y,E+f^*\Gamma,(\ba \sO_Y)^t)\cap \sO_X.
\] 
Let $\pi: X':=\Proj \bigoplus_{i \ge 0} \sO_X(\lfloor -i(K_X+D+\Gamma)\rfloor) \to X$ be the $\Q$-Cartierization of $-(K_X+D+\Gamma)$, that is, a projective birational morphism such that the strict transform of $-(K_X+D+\Gamma)$ is $\Q$-Cartier and ample.  
By Lemma \ref{lemma adjoint ideal for rings with f.g. anti-canonical rings}, 
we have 
\[
\adj_D(X,D+\Gamma,\ba^t)=\pi_*\adj_{D'}(X',D'+\Gamma',(\ba \sO_{X'})^t), 
\]
where $D'$ and $\Gamma'$ are the strict transforms on $X'$ of $D$ and $\Gamma$, respectively. 
Since $-(K_X'+D'+\Gamma')$ is ample, take an effective Cartier divisor $G'$ on $X'$ that is linearly equivalent to $-m(K_{X'}+D'+\Gamma')$ for sufficiently divisible integer $m \gg 0$. 
Consequently, $U_1=X' \setminus G'$ is an affine open subset of $X'$. 
As $G'$ varies, the corresponding $U_1$ cover $X'$. 
Therefore, it suffices to show that $c\in \adj_{D'|_{U_1}}(U_1,(D'+\Gamma')|_{U_1},(\ba\sO_{U_1})^t)$. 

Set $G:=\pi_*G'$, $U:=X\setminus (X_{\mathrm{sing}}\cup G)$ and  $V:=f^{-1}(U)\subseteq Y$, 
where $X_{\mathrm{sing}}$ is the singular locus of $X$. 
By \cite[Lemma 2.2]{Zhu}, $H^0(\sO_{U_1})=H^0(\sO_U)$ is a pure subring of $H^0(\sO_V)$. 
By setting $V_1:=\Spec H^0(\sO_V)$, this inclusion induces the morphism of affine varieties $g:V_1\to U_1$. 
Note that $g^*(D'|_{U_1})$ is a prime divisor on $V_1$ and $g^*(\Gamma'|_{U_1})$ has no component equal to $g^*(D'|_{U_1})$. 
Applying Lemma \ref{lemma finite generation}, one can find an effective $\R$-Weil divisor $\Delta_1$ on $V_1$ such that $\Delta_1$ and $g^*(D'|_{U_1})$ have no common components,  $K_{V_1}+g^*((D'+\Gamma')|_{U_1})+\Delta_1$ is $\Q$-Cartier and 
\[c\in \adj_{g^*(D'|_{U_1})}(V_1,g^*((D'+\Gamma')|_{U_1})+\Delta_1,(\ba \sO_{V_1})^t).\]
It follows from the definition of $U_1$ that $K_{U_1}+(D'+\Gamma')|_{U_1}$ is $\Q$-Cartier. 
Thus, by applying Corollary \ref{adjoint ideal under pure morphism in Q-Cartier case},  we have $c\in \adj_{D'|_{U_1}}(U_1,(D'+\Gamma')|_{U_1},(\ba \sO_{U_1})^t)$. 
\end{proof}

The following corollary generalizes Zhuang's result \cite[Theorem 2.10]{Zhu} and provides an affirmative answer to his question \cite[Question 2.13]{Zhu} in the klt case. 

\begin{cor}[{cf. \cite[Theorem 2.10]{Zhu}}]\label{cor plt case}
Let $f:Y\to X$ be a pure morphism between normal complex varieties, 
$D$ be a reduced divisor and $\Gamma$ be an effective $\Q$-Weil divisor on $X$ that has no common components with $D$.  
Suppose that the cycle-theoretic pullback $E=f^\natural D$ of $D$ under $f$ is a reduced divisor on $Y$. 
If $(Y,E+f^* \Gamma )$ is of plt type along $E$, then $(X,D+\Gamma)$ is of plt type along $D$. 
In particular, if $(Y, f^* \Gamma )$ is of klt type, then $(X,\Gamma)$ is of klt type as well.
\end{cor}

\begin{proof}
Given that $Y$ is of klt type, by a result of \cite{BCHM}, the $\sO_Y$-algebra $\bigoplus_{i\ge0} \sO_Y(iB)$ is finitely generated for every Weil divisor $B$ on $Y$. 
The $\sO_X$-algebras $\bigoplus_{i\ge0} \sO_X(\lfloor -i \Gamma \rfloor)$ and $\bigoplus_{i\ge0} \sO_X(\lfloor -i(K_X+\Gamma) \rfloor)$ are also finitely generated, as shown in \cite[Lemma 2.7]{Zhu}. 
Note that $f^*\Gamma$ is a $\Q$-Weil divisor due to the finite generation of the former graded ring.  
Since $(Y,E+f^*\Gamma)$ is of plt type along $E$, which forces $E$ to be supported on a disjoint union of prime divisors, 
it follows from Theorem \ref{adjoint ideals under pure ring extensions in rings with finitely generated anti-canonical algebra} that 
\[
 \adj_D(X,D+\Gamma) \supseteq f_*\adj_E(Y,E+f^*\Gamma) \cap \sO_X=f_*\sO_Y \cap \sO_X=\sO_X.  
\]
Thus, the pair $(X, D+\Gamma)$ is of plt type along $D$. 
\end{proof}

We conclude this section by focusing on the lc case. 
The following lemma gives a characterization of lc pairs in terms of multiplier ideals.
\begin{lem}[cf.~\textup{\cite[Lemma 1.3]{Tak04}}]\label{lem log canonical}
Let $X$ be a normal affine variety over an algebraically closed field of characteristic zero and let $\Gamma$ be an effective $\R$-Weil divisor on $X$.
Take a nonzero element $f$ of the multiplier ideal $\J(X,\Gamma)$. 
If $(X,\Gamma)$ is of lc type, then $f\in \J(X,\Gamma+(1-\epsilon)\Div (f))$ for every $0<\epsilon < 1$. 
When $K_X+\Gamma$ is $\Q$-Cartier, the converse also holds. 
\end{lem}
\begin{proof}
Suppose that $(X,\Gamma)$ is of lc type. 
By definition, there exist effective $\R$-Weil divisors $\Delta_1$ and $\Delta_2$ on $X$ such that $K_X+\Gamma+\Delta_i$ is $\R$-Cartier for each $i=1, 2$, $(X, \Gamma+\Delta_1)$ is lc and $f \in \mathcal{J}(X, \Gamma+\Delta_2)$. 
For each $0 < \epsilon <1$, 
set $\Delta_{\epsilon}=(1-\epsilon)\Delta_1+\epsilon  \Delta_2$.
Note that $K_X+\Gamma+\Delta_{\epsilon}$ is $\R$-Cartier. 
Let $\mu:Y\to X$ be a log resolution of the pair $(X, \Gamma+\Delta_1+\Delta_2+\Div (f))$. 
The containment $f\in \J(X, \Gamma+\Delta_2)$ implies that 
\[
\ord_{E}\left(K_Y-  \mu^*(K_X+\Gamma+\Delta_2)\right)+\ord_E(f)>-1. 
\]
On the other hand, since $(X, \Gamma+\Delta_1)$ is lc, we have $\ord_{E}\left(K_Y-\mu^*(K_X+\Gamma+\Delta_1)\right) \ge -1$ for all $i$. 
Therefore, for all $0<\epsilon<1$, one has 
\[
\ord_E\left(K_Y-\mu^*(K_X+\Gamma+\Delta_{\epsilon})\right)+\epsilon \ord_E(f)
>-(1-\epsilon)-\epsilon=-1,
\]
which is equivalent to saying that 
\[f\in J(X, \Gamma+\Delta_{\epsilon}+(1-\epsilon)\Div (f)) \subseteq \J(X, \Gamma+(1-\epsilon)\Div (f)).\] 

If $K_X+\Gamma$ is $\Q$-Cartier, then the above argument can be reversed by assuming $\Delta_1=\Delta_2=0$, and the latter assertion follows. 
\end{proof}

As an application of Corollary \ref{adjoint ideal under pure morphism in Q-Cartier case}, which has its roots in \cite[Theorem 1.2]{Yam}, 
we can provide a partial affirmative answer to another question posed by Zhuang \cite[Question 2.11]{Zhu} (cf.~\cite[Question 8.5]{BGLM}).  

\begin{thm}\label{lc case}
Let $\phi:Y\to X$ be a pure morphism between normal complex varieties and let $\Gamma$ be an effective $\Q$-Weil divisor on $X$ such that $K_X+\Gamma$ is $\Q$-Cartier. 
Assume that one of the following conditions holds. 
\begin{enumerate}[label=(\roman*)]
\item There exists an effective $\R$-Weil divisor $\Delta$ on $Y$ such that $K_Y+\phi^*\Gamma+\Delta$ is $\Q$-Cartier and no non-klt center of $(Y, \phi^*\Gamma+\Delta)$ dominates $X$. 
\item The non-klt-type locus of $(Y,\phi^*\Gamma)$ has dimension at most one. 
\end{enumerate}
If $(Y,\phi^*\Gamma)$ is of lc type, then $(X,\Gamma)$ is lc.  
\end{thm}

\begin{proof} 
Since the question is local, we may assume that $X$ and $Y$ are both affine. 
First, we consider the case (i).
This condition is equivalent to saying that $\J(Y,\phi^*\Gamma) \cap \sO_X \ne 0$, and therefore, we take a nonzero element $f \in \J(Y,\phi^*\Gamma) \cap \sO_X$. 
Since $(Y,\phi^*\Gamma)$ is of lc type, by Lemma \ref{lem log canonical}, $f$ lies in $\J(Y, \phi^*\Gamma+(1-\epsilon) \Div_Y (f))$ for all $0<\epsilon<1$. 
Applying Corollary \ref{adjoint ideal under pure morphism in Q-Cartier case}, we have 
\[
f \in \J(Y, \phi^*\Gamma+(1-\epsilon) \Div_Y (f)) \cap \sO_X \subseteq \J(X, \Gamma+(1-\epsilon) \Div_X (f))
\]
for all $0<\epsilon<1$. 
Then, by Lemma \ref{lem log canonical} once again, we can conclude that $(X, \Gamma)$ is lc. 

Next, we turn our attention to the case (ii). If $\J(Y,\phi^*\Gamma) \cap \sO_X \ne 0$, then we can reduce this to the case (i). 
Thus, we may assume that $\J(Y,\phi^*\Gamma) \cap \sO_X=0$. 
Since $\J(Y,\phi^*\Gamma)$ defines the non-klt-type locus $Z$ of $(Y,\phi^*\Gamma)$, it follows that $Z$ dominates $X$. 
However, by assumption, $Z$ has dimension at most one, which forces $X$ to also have dimension at most one. 
Therefore, $X$ is smooth and $\Supp \Gamma$ is a disjoint union of smooth prime divisors. 
As $(Y, \varphi^* \Gamma)$ is of lc type, the coefficients of $\Gamma$ are at most one, and in particular, the pair $(X, \Gamma)$ is lc. 
\end{proof}

\end{document}